\documentclass[10pt,a4paper,english]{article}
\usepackage[utf8]{inputenc}
\usepackage[english]{babel}
\usepackage{amsmath}
\usepackage{amsfonts}
\usepackage{amssymb}
\usepackage{amsthm}
\usepackage{pgfplots}
\usepackage{graphicx}
\usepackage{circuitikz}
\usepackage[justification=centering]{caption}
\usepackage[fixlanguage]{babelbib}
\graphicspath{ {./figures/} }
\usepackage{geometry}
\usepackage{placeins}
\usetikzlibrary{intersections, pgfplots.fillbetween}
\usetikzlibrary{patterns}
\usetikzlibrary{shapes,snakes}
\pgfplotsset{compat=1.16}
\newtheorem{Theorem}{Theorem}[section]
\newtheorem{Definitio}[Theorem]{Definition}
\newenvironment{Definition}{\begin{Definitio} \rm }{\end{Definitio}}

\newtheorem{Proposition}[Theorem]{Proposition}
\newtheorem{Remar}[Theorem]{Remark}
\newenvironment{Remark}{\begin{Remar} \rm }{\end{Remar}}
\newtheorem{Exampl}[Theorem]{Example}
\newenvironment{Example}{\begin{Exampl} \rm }{\end{Exampl}}
\author{Rudy Dissler}
\title{Relative Trisections of Fiber Bundles over the Circle}
\begin{document}
\maketitle
\begin{abstract}
For an oriented $4$--dimensional fiber bundle over $S^{1}$, we build a relative trisection from a sutured Heegaard splitting of the fiber. We provide an algorithm to explicitly construct the associated relative trisection diagram, from a sutured Heegaard diagram of the fiber. As an application, we glue our relative trisection diagrams with existing diagrams to recover trisected closed fiber bundles over $S^1$ and trisected spun manifolds, and to provide trisections for $4$--dimensional open-books.
\end{abstract}
\section{INTRODUCTION}
Gay and Kirby's trisection theory describes closed smooth $4$--manifolds as unions of simple pieces: $4$--dimensional $1$--handlebodies glued together in a suitable way. Results of existence and uniqueness (up to a stabilisation move) were established by Gay and Kirby, who also started to decline the concept in a relative setting, giving birth to the notion of relative trisections of compact $4$--manifolds \cite{gay2016trisecting}. The theory of relative trisections was then developed by Castro, Gay and Pinz\'{o}n-Caicedo in \cite{castro2018diagrams} and \cite{castro2018trisections}, establishing, as in the closed case, results of existence and uniqueness (up to stabilisation). Examples of closed trisected manifolds in the literature range from classic low-genus trisections in \cite{gay2016trisecting} or \cite{meier2017genus}, to more complicated constructions, such as trisections of fiber bundle over the circle by Koenig \cite{koenig2017trisections}, surface bundles over surfaces by Williams \cite{williams2020trisections}, or spun mani\-folds by Meier \cite{meier2017trisections}. In the relative setting, one can find simple examples such as relative trisections of $B^{4}$ or $I \times S^{3}$ in \cite{castro2016relative}, or more complicated examples of disk bundles over the $2$--sphere in \cite{castro2018diagrams}, or surface complements in \cite{kim2020trisections}. 
In both settings, these concrete examples are essential: not only do they build one's intuition, they also provide answers to pending questions (for instance, Meier's work was used later on in \cite{islambouli2021nielsen} to produce trisections of equal genus of the same manifold that actually need stabilisations before becoming isotopic). The purpose of this article is to construct another class of relative trisections: relative trisections of compact fiber bundles over the circle, in echo of Koenig's constructions in the closed setting \cite{koenig2017trisections}.

The interest of our constructions is two-fold. First, we build our relative trisections in a very elementary way, inspired from \cite{koenig2017trisections}, using a sutured Heegaard splitting of the fiber. It is therefore a rather good exercise in manipulating relative trisections, as it shows the complexity resulting from a non-empty boundary and allows to visualize quite precisely what is happening. Second, relative trisections can be glued together, provided that the induced decompositions of the boundaries agree \cite{castro2017trisections}. This gives rise to a (relative) trisection of the resulting space. This fact is used in \cite{castro2017trisections} or \cite{kim2020trisections} to build trisections of closed manifolds by gluing relative trisections of their different pieces. In this article, we will use our relatively trisected fiber bundles over the circle to construct trisections of closed fiber bundles over the circle, $4$--dimensional open books, and spun manifolds.

Let's focus on our bundle: the base is a circle and the fiber is a compact $3$--manifold with non-empty boundary. The main idea is to consider a sutured Heegaard splitting of the fiber that is either preserved or flipped by the monodromy of the bundle (which roughly means that the monodromy sends each compression body of the splitting to itself in the first case, and flips the compression bodies in the second). From this specific sutured Heegaard splitting, we will derive a relative trisection of the bundle, by combining the decomposition of the fiber given by the splitting with a decomposition of the base into intervals, using a technique of tubing. We will also show that there always exists a sutured Heegaard splitting of the fiber that is preserved by the monodromy. Therefore, there always exists a preserved sutured Heegaard splitting from which to derive a relative trisection; however, we also consider --- and actually start with --- the case of a flipped splitting because it is simpler and produces a lower genus trisection. The boundary of the fiber can be disconnected, but the action of the monodromy on its components complicates the computation of the genus of the relative trisection. These results are condensed in the following theorems.
\begin{Theorem}
\label{th:sut}
Let $M$ be a smooth, compact, oriented, connected $3$--manifold, and $\phi$ a self-diffeomorphism of $M$. Then $M$ admits a sutured Heegaard splitting $M=C_{1} \cup_{S} C_{2}$ that is preserved by $\phi$, i.e. there is an ambient isotopy of $M$ that sends $C_{1}$ to $\phi(C_{1})$ and $C_{2}$ to $\phi(C_{2})$.
\end{Theorem}

\begin{Remark}
Theorem \ref{th:sut} can be proven using triangulations, but we will stick to the differential category in our proof through Morse functions. 
\end{Remark}

\begin{Theorem}
\label{th:1flip}
Let $M$ be a smooth, compact, oriented, connected $3$--manifold and $\phi$ an orientation-preserving self-diffeomorphism of $M$. Let $X$ be the smooth oriented bundle over the circle with fiber $M$ and monodromy $\phi$. Suppose that $M$ admits a sutured Heegaard splitting $M=C_{1} \cup_{S} C_{2}$ that is flipped by $\phi$, with $S$ a compact surface with $b$ boundary components and Euler characteristic $\chi(S)$. Then $X$ admits a relative trisection $X=X_{1} \cup X_{2} \cup X_{3}$, where the Euler characteristic of $X_{1} \cap X_{2} \cap X_{3}$ is $3(\chi(S) -b)$. 
\end{Theorem}

\begin{Theorem}
\label{th:2pres}
Let $M$ be a smooth, compact, oriented, connected $3$--manifold and $\phi$ an orientation-preserving self-diffeomorphism of $M$. Let $X$ be the smooth oriented bundle over the circle with fiber $M$ and monodromy $\phi$. Then $M$ admits a sutured Heegaard splitting that is preserved by $\phi$. Given such a splitting, we denote by $S$ its Heegaard surface, with $b$ boundary components and Euler characteristic $\chi(S)$. Then $X$ admits a relative trisection $X=X_{1} \cup X_{2} \cup X_{3}$, where the Euler characteristic of $X_{1} \cap X_{2} \cap X_{3}$ is $6(\chi(S) -b)$. 
\end{Theorem}

The paper is organised as follows. Section 2 reviews the notions of sutured Heegaard splittings of compact $3$--manifolds, relative trisections of compact $4$--manifolds, as well as the associated diagrams. For a more thorough approach, the reader may refer to  \cite{castro2018diagrams} and \cite{castro2018trisections}. In Section 3, we make a detour in dimension $3$ to prove Theorem \ref{th:sut}, which shows that we can always find a sutured Heegaard splitting of the fiber that is preserved, up to isotopy, by the monodromy. In section $4$, we build our relative trisections of fiber bundles over the circle, provided that the fiber admits a sutured Heegaard splitting either preserved or flipped by the monodromy of the bundle. We compute the parameters of these relative trisections, giving more precise versions of Theorems \ref{th:1flip} and \ref{th:2pres}. In Section 5, we describe the corresponding diagrams. In Section 6, we use a gluing technique to construct trisection diagrams of some classes of closed manifolds from our relative trisection diagrams. First, we recover the trisected closed fiber bundle over the circle in \cite{koenig2017trisections}, then the trisected spun manifolds in \cite{meier2017trisections}; finally we provide a new class of trisections: trisected $4$--dimensional open-books, whose fiber is a $3$--manifold with boundary a torus. More specifically, we explicitly derive a $(6g+4)$--trisection diagram for the open-book from a genus $g$ sutured Heegaard diagram of the fiber.

\section{BACKGROUND}
We denote by $F_{g,b}$ a compact connected surface of genus $g$ with $b$ boundary components. We say that two decompositions of a manifold $M= \cup M_{i}$ and $M= \cup M'_{i}$ are isotopic if there is an ambiant isotopy of $M$ taking each $M_{i}$ to $M'_{i}$.
\begin{Definition}
\label{def:comp}
Let $S$ be a compact connected surface with non-empty boundary. Take the product $S \times [ 0 , 1 ]$ and add $3$--dimensional $2$--handles along a family $\alpha$ of disjoint, non trivial simple closed curves on $S \times \lbrace 0 \rbrace$. Call $C$ the resulting compact, connected $3$--manifold. We say that $C$ is a $\emph{compression body}$ and define:
\begin{itemize}
\item its positive boundary $\partial_{+}C$ as $S \times \lbrace 1 \rbrace$;
\item its negative boundary $\partial_{-}C$ as $\partial C \setminus \big( (\partial S \times \left] 0,1 \right[) \cup \partial_{+}C \big)$.
\end{itemize} 
We say that the set $\alpha$ is a $\emph{cut system}$ of $S$ corresponding to $C$.
\end{Definition}
Thus $\partial_{-}C$ is obtained by compressing $\partial_{+}C$ along the $\alpha$ curves. Throughout this article, we will only consider negative boundaries with no closed components. In this case, a compression body is a standard handlebody with a fixed decomposition of its boundary.
\begin{Remark}
We can also reverse the construction and obtain a compression body from its negative boundary $\partial_{-}C=F$, a compact (possibly disconnected) surface, by taking the product $F \times [ 0 , 1 ]$ and adding $3$--dimensional $1$--handles along $F \times \lbrace 1 \rbrace$ so that the resulting $3$--manifold $C$ is connected. In this perspective the negative boundary is $\partial_{-}C = F \times \lbrace 0 \rbrace$ and the positive boundary $\partial_{+}C$ is set as the closure of $\partial C \setminus \big( (F \times \lbrace 0 \rbrace) \cup (\partial F \times [0,1]) \big)$ (see Figure \ref{compbdypic}).
\end{Remark}
\begin{figure}[h!]
\[
\begin{tikzpicture}[scale=0.6]
\node (mypic) at (0,0) {\includegraphics[scale=0.8]{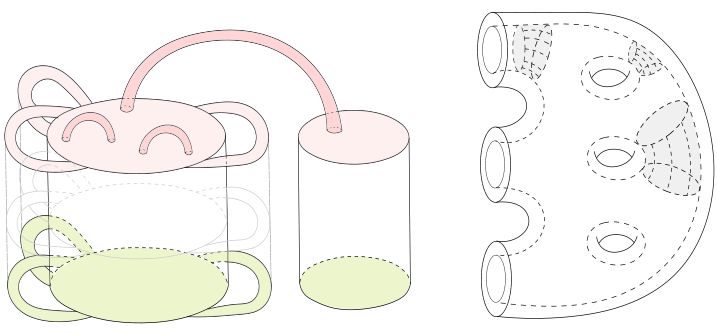}};
\end{tikzpicture}
\]
\caption{Two representations of a compression body\\{\footnotesize  Left: negative boundary in green and positive boundary in pink; $1$--handles in dark pink. Right: positive boundary in plain and negative boundary in dashed; $2$--handles in grey}}
\label{compbdypic}
\end{figure}

We will use handlebodies with a compression body structure to produce decompositions of compact $3$--manifolds with non-empty boundary. For that purpose, we need the concept of sutured manifold, introduced by Gabai in \cite{gabai1983foliations}. In what follows, we just keep the necessary amount of theory.   

\begin{Definition}
\label{def:shs}
Let $M$ be a compact, smooth, connected $3$--manifold with non-empty boundary. 
A sutured Heegaard splitting of $M$ is a decomposition $M=C_{1} \cup_{S} C_{2}$, where:
\begin{itemize}
\item the intersection $S=C_{1} \cap C_{2}$ is a compact connected surface with non-empty boundary;
\item $C_{1}$ and $C_{2}$ are handlebodies with a compression body structure, such that $\partial_{+}C_{1} = - \partial_{+}C_{2} = S$. 
\end{itemize}
 We say that a sutured Heegaard splitting of $M$ is \textit{balanced} if $\partial_{-}C_{1} \simeq \partial_{-}C_{2}$. 
\end{Definition}
\begin{Remark}
A sutured Heegaard splitting of $M$ induces a decomposition of $\partial M$ into two compact surfaces with non-empty boundary $S_{1}$ and $S_{2}$ and a finite set of annuli $\cup_{i} A_{i}$ joining each component of $\partial S_{1}$ to a component of $\partial S_{2}$. We call such a decomposition a $\emph{sutured decomposition}$ of $\partial M$.
\end{Remark}
\begin{Remark}
Necessarily $\cup_{i} A_{i} \simeq ( \partial S_{1} \times [0,1] ) \simeq ( \partial S_{2} \times [0,1] )$. A curve on a $A_{i}$ isotopic to a component of $\partial S_{1}$ is called a suture of the decomposition. Sometimes in this article, we will not need to differentiate a sutured decomposition of a surface $S$ from a decomposition into two compact surfaces glued along their common boundary. When that is the case, we will also refer to the latter as a sutured decomposition of $S$.
\end{Remark}
We now move to dimension 4. We will just briefly define relative trisections: for proofs and more details, we refer to \cite{castro2016relative}, \cite{castro2018diagrams} and \cite{castro2018trisections}.
\begin{Definition}
\label{def:relative}
A relative trisection of a compact, connected, smooth $4$--manifold $X$ is a decomposition $X=X_{1} \cup X_{2} \cup X_{3}$ such that:
\begin{itemize}
\item each $X_{i}$ is a $4$--dimensional handlebody, i.e. $X_{i} \simeq \natural^{k_{i}} (S^{1} \times B^{3})$ for some $k_{i}$;
\item the triple intersection $X_{1} \cap X_{2} \cap X_{3} = S$ is a compact connected surface with boundary;
\item each intersection $\big( X_{i} \cap X_{j} \big) \cap \partial X$ is diffeomorphic to a given compact surface~$P$;
\item each double intersection $X_{i} \cap X_{j} = (\partial X_{i} \cap \partial X_{j})$ is a handlebody $C_{i,j}$ with a compression body structure defined by $\partial_{+}C_{i,j} = S$ and $\partial_{-}C_{i,j} = \big( X_{i} \cap X_{j} \big) \cap \partial X \simeq P$;
\item each $X_{i} \cap \partial X$ is diffeomorphic to $P \times I$.
\end{itemize}
\end{Definition}  
Definition \ref{def:relative} induces a decomposition of $\partial X$ into:
\begin{itemize}
\item three sets diffeomorphic to $P \times I$, glued together to form a $P$--fiber bundle over the circle;
\item solid tori $\partial P \times I \times I$ glued trivially so as to fill each boundary component of the bundle. 
\end{itemize}
This is an open-book decomposition of $\partial X$, with page $P$ and binding $\partial P \simeq \partial S$.

We call the triple intersection the $\emph{central surface}$ of the trisection. If the boun\-dary of $X$ is connected, a $(g,k_{1},k_{2},k_{3},p,b)$--relative trisection is a relative trisection involving a central surface of genus $g$ with $b$ boundary components, $4$--dimensional handlebodies of genus $k_{i}$ and pages of genus $p$ (if all the $k_{i}$'s are equal to $k$, we will simply write a $(g,k,p,b)$--relative trisection). If $\partial X$ is not connected, then there are as many pages as components in $\partial X$. In this case we will use multi-indices as $p$ and $b$, and notice that the sum of the indices in $b$ must be equal to the number of boundary components of the central surface.  

One key feature of sutured Heegaard splittings and relative trisections is that they can both be described by diagrams, just as Heegaard splittings and trisections. 
\begin{Definition}
A sutured Heegaard diagram is a triple $(S, \alpha, \beta)$, where $S$ is a compact connected surface with boundary and $\alpha$ and $\beta$ are two sets of disjoint, non-trivial simple closed curves in $Int(S)$, such that compressing $S$ along each set does not produce any closed component.   
\end{Definition} 
Each set thus corresponds to a cut system for some compression body, and the last condition ensures that this compression body is also a handlebody. Therefore, a sutured Heegaard diagram defines a $3$--manifold with a sutured Heegaard splitting. We can reverse the process and obtain a sutured Heegaard diagram from a sutured Heegaard splitting, by setting $S$ as the intersection $C_{1} \cap C_{2}$, the $\alpha$'s (resp. the $\beta$'s) as a cut system defining $C_{1}$ (resp. $C_{2}$). Actually, there is a one-to-one correspondence between sutured Heegaard diagrams (up to diffeomorphism of the surface and handleslides of curves within each cut system) and sutured Heegaard splittings (up to diffeomorphism). 
\begin{Example} 
Given a Heegaard diagram $(\Sigma, \alpha, \beta )$ (associated to a closed $3$--manifold $M$), consider the sutured Heegaard diagram $(S, \alpha, \beta)$, where $S$ is $\Sigma$ minus the interior of $b$ disks disjoint from the $\alpha$ and $\beta$ curves. If $b=1$, we obtain a sutured Heegaard diagram associated to $M$ minus the interior of a $3$--ball; if $b=2$, we obtain a sutured Heegaard diagram associated to $M$ minus the interior of a solid torus. If $b=2$ and we add one curve parallel to a boundary component to the $\alpha$'s or the $\beta$'s, we obtain again $M$ minus the interior of a $3$--ball; if $b=2$ and we add to each set a curve parallel to a boundary component, the result is $M$ minus the interior of two disjoint $3$--balls. 
\end{Example}
\begin{Definition}
A standard sutured Heegaard diagram is the connected sum of standard genus 1 Heegaard diagrams of $S^{3}$ and $S^{1} \times S^{2}$ and a compact surface with non-empty boundary $F_{p,b}$.
\end{Definition}
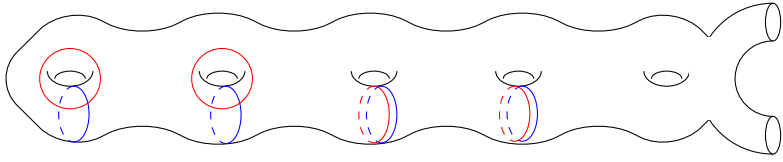
\begin{figure}[h!]
\[
\begin{tikzpicture}
\draw[rounded corners=0.5cm] (0,0) -- (1,1) -- (2,0.5) -- (3,1) -- (4,0.5) -- (5,1) -- (6,0.5) -- (7,1) -- (8,0.5) -- (9,1) -- (10,0) -- (9,-1) -- (8,-0.5) -- (7,-1) -- (6,-0.5) -- (5,-1) -- (4,-0.5) -- (3,-1) -- (2,-0.5) -- (1,-1) -- cycle ;
\draw (0.75,0.1) arc (-180:0:0.3cm and 0.2cm);
\draw (0.85,0) arc (180:0:0.2cm and 0.1cm);
\draw (2.75,0.1) arc (-180:0:0.3cm and 0.2cm);
\draw (2.85,0) arc (180:0:0.2cm and 0.1cm);
\draw (4.75,0.1) arc (-180:0:0.3cm and 0.2cm);
\draw (4.85,0) arc (180:0:0.2cm and 0.1cm);
\draw (6.65,0.1) arc (-180:0:0.3cm and 0.2cm);
\draw (6.75,0) arc (180:0:0.2cm and 0.1cm);
\draw [blue] (1.1,-0.1) arc (90:-90:0.2cm and 0.37cm);
\draw [blue, dashed] (1.1,-0.1) arc (90:270:0.2cm and 0.37cm);
\draw [blue] (3.1,-0.1) arc (90:-90:0.2cm and 0.38cm);
\draw [blue, dashed] (3.1,-0.1) arc (90:270:0.2cm and 0.38cm);
\draw [blue] (5.15,-0.1) arc (90:-90:0.2cm and 0.375cm);
\draw [blue, dashed] (5.15,-0.1) arc (90:270:0.2cm and 0.375cm);
\draw [red] (5.05,-0.1) arc (90:-90:0.2cm and 0.375cm);
\draw [red, dashed] (5.05,-0.1) arc (90:270:0.2cm and 0.375cm);
\draw [blue] (7,-0.1) arc (90:-90:0.2cm and 0.365cm);
\draw [blue, dashed] (7,-0.1) arc (90:270:0.2cm and 0.365cm);
\draw [red] (6.9,-0.1) arc (90:-90:0.2cm and 0.365cm);
\draw [red, dashed] (6.9,-0.1) arc (90:270:0.2cm and 0.365cm);
\draw [red] (1.05,0) circle (0.4cm);
\draw [red] (3.05,0) circle (0.4cm);
\draw (8.6,0.1) arc (-180:0:0.3cm and 0.2cm);
\draw (8.7,0) arc (180:0:0.2cm and 0.1cm);
\draw (10.3,0) circle (1cm);
\fill[white] (9.2,-0.55) -- (10.5,-0.55) -- (10.5,0.55) -- (9.2,0.55) -- (9.2,-0.55) -- cycle;
\draw (10.3,0) circle (0.5cm);
\fill[white] (10.3,-1.5) -- (12,-1.5) -- (12,1.5) -- (10.3,1.5) -- (10.3,-1.5) -- cycle;
\draw (10.3,1) arc (90:450:0.1cm and 0.25cm);
\draw (10.3,-0.5) arc (90:450:0.1cm and 0.25cm);
\end{tikzpicture}
\]
\caption{A standard sutured Heegaard diagram} 
\label{fig:hdYP}
\end{figure} 
\begin{Definition}
A relative trisection diagram is a $4$--tuple $\lbrace S, \alpha, \beta, \gamma \rbrace$, where each triple $\lbrace S, \alpha, \beta \rbrace$, $\lbrace S, \beta, \gamma \rbrace$, $\lbrace S, \alpha, \gamma \rbrace$ is handleslide-diffeomorphic to a standard sutured Heegaard diagram involving the same compact surface $P \simeq F_{p,b}$. 
\end{Definition}
There is a one-to-one correspondence between relative trisection diagrams (up to handleslides of curves within each family $\alpha$, $\beta$ and $\gamma$ and diffeomorphism of $S$) and relatively trisected $4$--manifolds (up to diffeomorphism), see \cite{castro2018diagrams}. If a relative trisection diagram corresponds to a relative trisection $X=X_{1} \cup X_{2} \cup X_{3}$, each sutured diagram consisting of the trisection surface and two sets of curves corresponds to a sutured Heegaard diagram of one $\big( X_{i-1} \cap X_{i} \big) \cup \big( X_{i} \cap X_{i+1} \big)$.

\section{A monodromy-preserved decomposition of M}  
We want to show that no assumption has to be made on $M$ to find a sutured Heegaard splitting that is preserved by the monodromy. This is the content of Theorem \ref{th:sut}, which we are now going to prove. 
To do so, we need some preliminary results.
\begin{Proposition}
\label{prop:exsplit}
Given a sutured decomposition of $\partial M$, there exists a sutured Heegaard splitting $M=C_{1} \cup_{S} C_{2}$ inducing it.
\end{Proposition}
\begin{proof} 
We will use a handle decomposition given by a Morse function. We denote the sutured decomposition of the boundary of $M$ by $\partial M \simeq S_{1} \cup (\partial S_{1} \times [0,1] ) \cup S_{2}$ and consider a collar neighbourhood $\nu$ of $\partial M$, $\nu = \psi \big((S_{1} \cup (\partial S_{1} \times [0,1] ) \cup S_{2}) \times [0,1/2]\big)$, with $\psi$ an embedding such that $ \psi \big( (S_{1} \cup (\partial S_{1} \times [0,1] ) \cup S_{2}) \times \lbrace 0 \rbrace \big) = \partial M$. We define the following smooth function $f$ on $\partial M$:
\begin{itemize}
\item on $\psi(S_{1} \times \lbrace 0 \rbrace)$, $f \equiv -1$; on $\psi(S_{2} \times \lbrace 0 \rbrace)$, $f \equiv 4$;
\item for $((x,u),0) \in (\partial S_{1} \times [0,1]) \times \lbrace 0 \rbrace$, $f(\psi((x,u),0))=g(u)$, where g is a smooth, strictly increasing function from $[0,1]$ to $[-1,4]$ that verifies: $g(0)=-1$, $g(1)=4$, and $g^{(k)}(0)=g^{(k)}(1)=0$ for $k > 0$. 
\end{itemize}
Then we extend $f$ to a smooth function $\tilde{f}$ on $\nu$ by:
\begin{itemize}
\item $\tilde{f}(\psi(y,t))=t-1$ for $(y,t) \in S_{1} \times [0,1/2]$ and $\tilde{f}(\psi(y,t))=-t+4$ for $(y,t) \in S_{2} \times [0,1/2]$;
\item for $\big( (x,u),t \big) \in (\partial S_{1} \times [0,1]) \times [0,1/2]$, $\tilde{f}\big( \psi((x,u),t) \big) = t h(u) +g(u)$, where $h$ is a smooth strictly decreasing function from $[0,1]$ to $[-1,1]$ such that:
\begin{itemize}
\item $h(0)=1$; $h(1/2)=0$; $h(1)=-1$;
\item $h^{(k)}(0)=h^{(k)}(1)=0$ for $k > 0$;
\item $ - h'(u) < 2 g'(u)$ for $0<u<1$.
\end{itemize}
\end{itemize} 
Then $\tilde{f}$ is a smooth function with no critical point on $\nu$, that sends $\psi(\partial M \times \lbrace 1/2 \rbrace)$ to $[-0.5,3.5]$. As $\psi(\partial M \times \lbrace 1/2 \rbrace)$ is embedded in $M$, we can use slice charts to extend $\tilde{f}$ to $M$ and finally obtain a smooth function on $M$, with no critical points on $\nu$, that sends $M \setminus \nu$ to $[-0.5,3.5]$. We can use a generic Morse approximation of this function, then modify it around its critical points (therefore, away from $\nu$) to obtain a self-indexing Morse function $F$. Consider a critical point of index $0$. It corresponds to adding a $0$--handle, but, as $M$ is connected, there must be a cancelling $1$--handle. The same applies to the components of $F^{-1}(\lbrace -1 \rbrace)=S_{1}$, that must be connected by $1$--handles. Therefore, the preimage $F^{-1}([-1,3/2])=C_{1}$ is diffeomorphic to the connected union of $F^{-1}(\lbrace -1 \rbrace) \times I$ and $1$--handles: it is a handlebody with a compression body structure given by $\partial_{-}C_{1} = S_{1}$ and $\partial_{+}C_{1} =F^{-1}(\lbrace 3/2 \rbrace)=S$, where $S$ is a compact connected surface with boundary. By applying the same argument to $-F$, we see that $C_{2}=F^{-1}([3/2,4])$ is also a handlebody with a compression body structure given by $\partial_{-}C_{2} = S_{2}$ and $\partial_{+}C_{2} =  F^{-1}(\lbrace  3/2 \rbrace)=S$. Therefore $M=C_{1} \cup_{S} C_{2}$ is a sutured Heegaard splitting of $M$, inducing the requested sutured decomposition of $\partial M$. 
\end{proof}
\begin{Remark}
\label{rmkunbalanced}
These arguments are standardly used to prove the handle decomposition theorem and the existence of Heegaard splittings (see for instance \cite{kosinski2013differential}). We simply adapted them to the setting created by a sutured decomposition of the boundary. Notice that the gradient of $f$ can be made parallel to $\partial M$ in $\psi(\partial S_{1} \times [0,1])$ for some metric, which is important to understand that $F^{-1} ([-1,0])$ is a thickening of $S_{1}$ (with perhaps some extra $3$--balls) respecting the sutures (see \cite{borodzik2016morse}). We also refer to \cite{juhasz2006holomorphic}: this latter article actually covers the proof, but we found it useful to give our version, as \cite{juhasz2006holomorphic} focuses on balanced sutured Heegaard diagrams. In this paper, we do not need the full diagrammatic approach of \cite{juhasz2006holomorphic}, but we do need to work with non balanced decompositions.
\end{Remark}
\begin{Definition}
A Morse function $f:M \rightarrow [-1,4]$ is \emph{compatible} with a sutured Heegaard splitting $M=C_{1} \cup_{S} C_{2}$ if:
\begin{itemize}
\item we have $f^{-1}([-1,3/2])=C_{1}$, $f^{-1}([3/2,4])=C_{2}$, and $f^{-1}(3/2)=S$;
\item the critical points of $f$ of index 0 and 1 belong to $C_{1}$ and  the critical points of index 2 and 3 belong to $C_{2}$.
\end{itemize} 
\end{Definition}
Now we prove that \emph{any} sutured Heegaard splitting can be given by a compatible Morse function. 
\begin{Proposition}
\label{prop:morsesut}
If $M=C_{1} \cup_{S} C_{2}$ is a sutured Heegaard splitting, then there is a Morse function compatible with it.
\end{Proposition}
\begin{proof}
We set $V_{1}= \partial_{-}C_{1} \setminus \nu_{1}$, where $\nu_{1}$ is a small regular neighbourhood of $\partial (\partial_{-} C_{1})$ in $\partial_{-} C_{1}$. Consider a handle $H$ glued to $V_{1} \times [0,1]$. We write $T$ a regular neighbourhood  of the attaching sphere of $H$ in $V_{1} \times [0,1]$. Then there is a Morse function $G$ on $( V_{1} \times [0,1]) \cup H$ such that $G(x,t)=t$ outside of the subset $H \cup T$, with only one critical point corresponding to the attachment of $H$. Considering the handles to be attached separately, we combine such functions into a Morse function $F$ on the union of $V_{1} \times [0,1]$ and of the $1$--handles of the splitting, with critical points of index $1$ corresponding to the attachment of the $1$--handles. We can apply the same argument to $C_{2}$ considered as the union of $\partial_{-}C_{2} \times I$ and $1$--handles, then take the opposite Morse function to produce a Morse function with critical points corres\-ponding to the $2$--handles, equal to $4$ on the intersection with $\partial M$ and agreeing with $F$ on $S$. Combining the two functions gives a Morse function corresponding to the handle decomposition on $C_{1} \cup C_{2}$ minus a regular neighbourhood of $\partial S \times I$. Finally we interpolate between the values on $\partial V_{1} \times [0,1]$ and the values defined on $\partial(\partial_{-} C_{1})$ for the Morse function of Proposition \ref{prop:exsplit} without creating new critical points, which produces the desired Morse function on $M$. 
\end{proof}
\begin{Proposition}
\label{prop:isot}
If two sutured Heegaard splittings of a manifold $M=C_{1} \cup_{S} C_{2}$ and $M=C_{1}' \cup_{S'} C_{2}'$ induce the same sutured decomposition of $\partial M$, then they become isotopic after a finite sequence of interior stabilisations. 
\end{Proposition}
\begin{proof}
Using Proposition \ref{prop:morsesut}, let $f$ (resp. $f'$) be a Morse function defining $M=C_{1} \cup_{S} C_{2}$ (resp. $M=C_{1}' \cup_{S'} C_{2}'$). We can assume that $f$ and $f'$ agree on a neighbourhood of $\partial M$ (see the proof of Proposition \ref{prop:morsesut}). Then Cerf theory produces a path of generalised ordered Morse functions $\lbrace f_{t}\rbrace_{t \in [0,1]}$ from $f=f_{0}$ to $f'=f_{1}$. As $f_{0}$ 
and $f_{1}$ do not have critical points on a regular neighbourhood of $\partial M$, we can assume that the $f_{t}$'s behave accordingly, and that they agree with $f_{0}$ and $f_{1}$ on this neighbourhood, therefore fixing the sutured decomposition of $\partial M$. Now we can apply standard arguments of Cerf theory and handle decompositions (see \cite{juhasz2006holomorphic}): as $t$ moves from $0$ to $1$, the $f_{t}$ will produce a finite number of births and deaths of critical points, which correspond to stabilisations and destabilisations, as well as a finite number of handleslides.
\end{proof}   
Now we need another proposition, which relies on more combinatorial arguments, and will allow us to conclude our proof of Theorem \ref{th:sut}. 

\begin{Proposition}
\label{prop:shs}
If $M$ is a compact oriented connected $3$--manifold with boundary, and $\phi$ is a self-diffeomorphism of $M$, then there exists a sutured decomposition of $\partial M$ that is preserved by a diffeomorphism of $M$ isotopic to $\phi$.
\end{Proposition}
\begin{proof}  
Pick a component $\partial_{1}M$ of $\partial M$, and choose a disk $D_{1} \subset \partial_{1}M$.
Now, let's consider the orbit of $\partial_{1} M$ under the action of $\phi$:
\begin{equation*}
\lbrace \partial_{1} M, \; ... \; ,\partial_{j} M = \phi(\partial_{j-1} M) = \phi^{j-1}(\partial_{1} M), \; ... \; ,\partial_{s} M= \phi^{s-1}(\partial_{1} M)= \phi^{-1}(\partial_{1}M) \rbrace. 
\end{equation*} 
Then $\phi^{s}(D_{1})$ is not necessarily $D_{1}$, but we can compose $\phi_{\mid \partial M}$ with a diffeomorphism $g$ of $\partial M$ equal to the identity on every component other than $\partial_{1} M$, and isotopic to the identity on $\partial_{1}M$ but sending $\phi^{s}(D_{1})$ to $D_{1}$. As $g$ is isotopic to the identity on $\partial M$, we can extend it to a diffeomorphism of $M$ isotopic to the identity, that we still call $g$. Then $g \circ \phi$ is a diffeomorphism of $M$ isotopic to $\phi$ that 
preserves the set $\lbrace D_{1}, D_{2}=\phi(D_{1}), \; ... \; , D_{s}=\phi^{s-1}(D_{1})\rbrace$. By setting a sutured decomposition of each component:
  \begin{equation*}
\partial_{j} M = D_{j} \cup (\partial D_{j} \times [0,1]) \cup \overline{ \big( \partial_{j} M \setminus ( D_{j} \cup (\partial D_{j} \times [0,1]) \big)}
\end{equation*}
we obtain a sutured decomposition of the orbit of $\partial_{1}M$ that is preserved by \mbox{$g \circ \phi$}. Now we can apply the same argument on each orbit of the action of $\phi$ on $\partial M$, as $g$ leaves these orbits invariant. We have thus produced the desired sutured decomposition of $\partial M$. 
\end{proof}
\begin{proof}[Proof of Theorem \ref{th:sut}]
Let $\phi$ be a self-diffeomorphism of $M$. According to Proposition \ref{prop:shs}, we may assume that there is always a sutured decomposition of $\partial M$ that is preserved by $\phi$ (to be precise, an unbalanced sutured decomposition, see Remark \ref{rmkunbalanced}). Then by Proposition \ref{prop:exsplit} there is a sutured Heegaard splitting $M=C_{1} \cup_{S} C_{2}$ inducing this decomposition. Now we can transpose to the sutured setting the argument of \cite{koenig2017trisections}: the image $M=\phi(C_{1}) \cup_{\phi(S)} \phi(C_{2})$ is also a sutured Heegaard splitting of $M$, involving the same decomposition of $\partial M$. Using Proposition \ref{prop:isot} (and the fact that the image of a stabilisation is a stabilisation of the image), we can conclude that, after a number of interior stabilisations, we always obtain a sutured Heegaard splitting that is preserved by $\phi$, up to isotopy. 
\end{proof}

\section{Constructing the relative trisections}
Given a compact, oriented, connected, smooth $3$--manifold $M$ with non-empty bounda\-ry (possibly disconnected), together with an orientation preserving self-diffeomor\-phism $\phi$, we consider the fiber bundle $X=\big( M \times I \big) / \sim$, where $(x,0) \sim (\phi(x) , 1)$. Then $X$ is a smooth, compact, oriented, connected $4$--manifold (with non-empty boundary), of which we want to build a relative trisection.

We suppose that $M$ admits a sutured Heegaard splitting $M = C_{1} \cup_{S} C_{2}$ that $\phi$ preserves (i.e. $\phi(C_{i}) = C_{i}$) or flips (i.e. $\phi(C_{i}) = C_{i+1}$, considering indices modulo two). We denote by $g$ the genus of $S$ and by $b$ its number of boundary components, i.e. $S \simeq F_{g,b}$. According to Theorem \ref{th:sut}, we can always expect $M$ to admit a sutured Heegaard splitting that is preserved, up to isotopy, by $\phi$ (the fact that it is only preserved up to isotopy is not a problem because isotopic monodromies produce diffeomorphic bundles).
\begin{Remark}
The monodromy acts on the components of both $\partial M$ and $\partial S$; in either case it can preserve or flip boundary components. We don't make any assumptions on how the monodromy acts on the boundary components in what follows, unless otherwise stated.
\end{Remark}
\subsection{Case 1: $\phi$ flips a sutured Heegaard splitting of $M$}
In this case, necessarily $\partial_{-}C_{1} \simeq \partial_{-}C_{2}$, i.e. $M=C_{1} \cup_{S} C_{2}$ is balanced. We can re\-present $X$ as a rectangle with vertical edges identified, as in Figure \ref{schemX}. The horizontal segment corresponds to the interval $I$ parametrized by $t$. Each vertical segment is a copy of $M$. The vertical segments $M \times \lbrace 0 \rbrace$ and $M \times \lbrace 1 \rbrace$ are identified according to $\phi$.
\begin{figure}[h!]
\[
\begin{tikzpicture}[scale=0.7]
\draw (0,0) -- (8,0) -- (8,4) -- (0,4) -- (0,0);
\draw [dashed, stealth-] (0,-0.1) -- (0,-1);
\draw [dashed, stealth-] (8,-0.1) -- (8,-1); 
\node (a) at (0,4.28) {0};
\node (a) at (8,4.28) {1};
\node (a) at (4,-0.7) {gluing along $\phi$};
\draw [dotted] (3,0) -- (3,1.7);
\draw [dotted] (3,4) -- (3,2.3);
\draw [dashed] (0,-1) -- (8,-1);
\node (a) at (3,2) {$M \times \lbrace t \rbrace$};
\end{tikzpicture}
\]
\caption{A representation of $X$}
\label{schemX}
\end{figure}
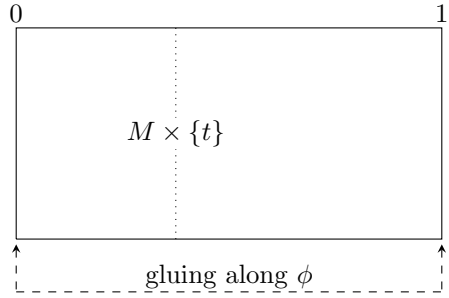 
\\
Next we divide each copy $M \times \lbrace t \rbrace$ according to $\big( C_{1} \cup_{S} C_{2} \big) \times \lbrace t \rbrace$. We further divide the horizontal segment $I$, keeping in mind that $C_{i}$ is identified with $C_{i+1}$ at $0 \sim 1$, to obtain the decomposition represented on Figure \ref{schemX1}.
 \begin{figure}[h!]
\[
\begin{tikzpicture}
\draw (0,0) -- (8,0) -- (8,4) -- (0,4) -- (0,0);
\draw [dashed] (0,0) -- (0,-0.5);
\draw [dashed] (8,0) -- (8,-0.5); 
\node (a) at (0,4.2) {0};
\node (a) at (8,4.2) {1};
\node (a) at (0,-0.7) {$\phi$};
\node (a) at (8,-0.7) {$\phi$};
\draw [very thick] (0,2) -- (8,2);
\draw [stealth - stealth] (8.5,4) -- (8.5,2.5);
\draw [stealth - stealth] (8.5,0) -- (8.5,1.5); 
\node (b) at (8.5,2) {$S$};
\node (b) at (8.8,1) {$C_{1}$};
\node (b) at (8.8,3) {$C_{2}$};
\draw (2.66,0) -- (2.66,2);
\node (a) at (2.66,4.2) {$t_{1}$};
\draw (4,2) -- (4,4);
\node (a) at (4,4.2) {$t_{2}$};
\draw (5.33,2) -- (5.33,0);
\node (a) at (5.33,4.2) {$t_{3}$};
\node (c) at (0.4,3.5) {$X_{2} '$};
\node (c) at (7.6,0.5) {$X_{2} '$};
\node (c) at (7.6,3.5) {$X_{3} '$};
\node (c) at (0.4,0.5) {$X_{3} '$};
\node (c) at (4,0.5) {$X_{1} '$};
\end{tikzpicture}
\]

 \caption{A decomposition of $X$}
\label{schemX1}
\end{figure}
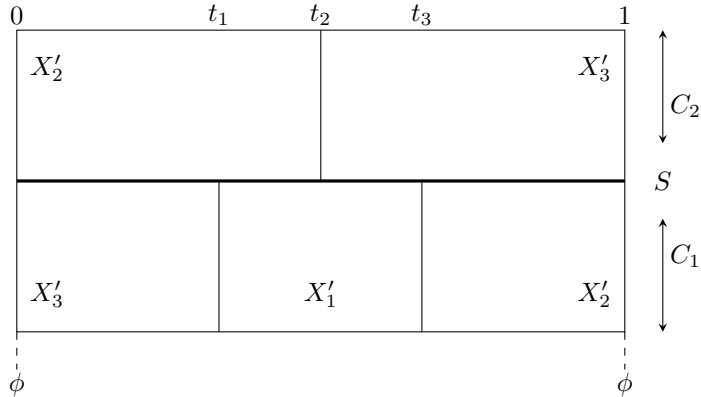

Now, each $X_{k}'$ on Figure \ref{schemX1} is diffeomorphic to a $C_{i} \times I$. Because the $C_{i}$'s are tridimensional handlebodies, the $X_{k}'$'s are $4$--dimensional handlebodies. But the triple intersection is not connected. So we still have a bit of work to do: the last step is to drill out three sets of $4$--dimensional tubes $I \times B^{3}$ along the boundary of the $X_{k}'$'s, then assign each set of tubes to one $X_{k}'$, giving the final decomposition of Figure \ref{trisXcas1}. The global construction is the same as in the closed case (\cite{koenig2017trisections}), but because of the boundary, we have to be more careful as to how we design our tubes.

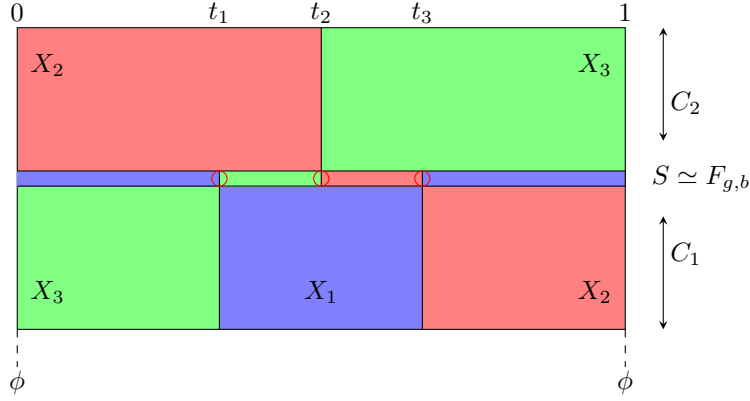
\begin{figure}[h!]
\[
\begin{tikzpicture}
\fill [green!50] (0,0) -- (2.66,0) -- (2.66,2) -- (0,2) -- (0,0);
\fill [green!50] (4,2) -- (8,2) -- (8,4) -- (4,4) -- (4,2);
\fill [red!50] (0,2) -- (4,2) -- (4,4) -- (0,4) -- (0,2);
\fill [red!50] (5.33,0) -- (8,0) -- (8,2) -- (5.33,2) -- (5.33,0);
\fill [blue!50] (2.66,0) -- (5.33,0) -- (5.33,2) -- (2.66,2) -- (2.66,0);

\draw (0,0) -- (8,0) -- (8,4) -- (0,4) -- (0,0);
\draw [dashed] (0,0) -- (0,-0.5);
\draw [dashed] (8,0) -- (8,-0.5); 
\node (a) at (0,4.2) {0};
\node (a) at (8,4.2) {1};
\node (a) at (0,-0.7) {$\phi$};
\node (a) at (8,-0.7) {$\phi$};
\fill [white] (0,1.9) -- (8,1.9) -- (8,2.1) -- (0,2.1) -- (0,2.1);
\fill [green!50] (2.66,1.9) -- (4,1.9) -- (4,2.1) -- (2.66,2.1) -- (2.66,1.9);
\fill [blue!50] (0,1.9) -- (2.66,1.9) -- (2.66,2.1) -- (0,2.1) -- (0,1.9);
\fill [blue!50] (5.33,1.9) -- (8,1.9) -- (8,2.1) -- (5.33,2.1) -- (5.33,1.9);
\fill [red!50] (4,1.9) -- (5.33,1.9) -- (5.33,2.1) -- (4,2.1) -- (4,1.9);
\draw (0,1.9) -- (8,1.9) -- (8,2.1) -- (0,2.1) -- (0,2.1);
\draw [stealth - stealth] (8.5,4) -- (8.5,2.5);
\draw [stealth - stealth] (8.5,0) -- (8.5,1.5); 
\node (b) at (9,2) {$S \simeq F_{g,b}$};
\node (b) at (8.8,1) {$C_{1}$};
\node (b) at (8.8,3) {$C_{2}$};
\draw (2.66,0) -- (2.66,2.1);
\node (a) at (2.66,4.2) {$t_{1}$};
\draw (4,1.9) -- (4,4);
\node (a) at (4,4.2) {$t_{2}$};
\draw (5.33,2.1) -- (5.33,0);
\node (a) at (5.33,4.2) {$t_{3}$};
\node (c) at (0.4,3.5) {$X_{2}$};
\node (c) at (7.6,0.5) {$X_{2}$};
\node (c) at (7.6,3.5) {$X_{3}$};
\node (c) at (0.4,0.5) {$X_{3}$};
\node (c) at (4,0.5) {$X_{1}$};
\draw [red] (2.66,2) circle (0.1cm);
\draw [red] (4,2) circle (0.1cm);
\draw [red] (5.33,2) circle (0.1cm);
\end{tikzpicture}
\]
\caption{A trisection of $X$, case 1}
\label{trisXcas1}
\end{figure}

Let's focus on the horizontal narrow coloured stripes dividing the rectangle in Figure~\ref{trisXcas1}. Every stripe represents a set of $b$ $4$--dimensional tubes, i.e. copies of $I \times B^{3}$. So each tube can be seen as a family of $3$--balls $\lbrace B_{t} \rbrace$, parametrized by $t$. We impose that, for each $t$, $B_{t}$ intersects $S \times \lbrace t \rbrace$. Now me must specify \textit{how}. For instance, if each $B_{t}$ intersects $S \times \lbrace t \rbrace$ on the interior of the surface, the pages $\big( X_{i} \cap X_{j} \big) \cap \partial X$ are never connected.
\begin{itemize}
\item We start with a tube $[t_{1},t_{2}] \times B^{3}$. We choose a component of $\partial S$ that we call $\partial_{1}S$. Then we define, for each $t \in [t_{1},t_{2}]$, a ball $B_{t}$ intersecting $S \times \lbrace t \rbrace$ along a disk, $\partial_{1}S \times \lbrace t \rbrace$ along a closed interval and $\partial M \times \lbrace t \rbrace$ along a disk. Thus we create a path of $3$--balls $\lbrace B_{t}^{3,1} \mid t \in [t_{1},t_{2}] \rbrace$. We impose that this path is smooth and set it as our first tube $T_{3,1}$. We define in the same fashion a tube $T_{2,1}$ in $M \times [t_{2},t_{3}]$.
\item Then we design a quotient tube $T_{1,1} = \big( [t_{3},1] \cup_{0 \sim 1} [0,t_{1}] \big) \times B^{3}$, by gluing two tubes $\lbrace B_{t}^{1,1} \mid t \in [t_{3},1] \rbrace$ and $\lbrace B_{t}^{1,1} \mid t \in [0,t_{1}] \rbrace$ at $0 \sim 1$. 
\item We impose that:
\begin{itemize}
\item all the tubes are disjoint;
\item $B_{t_{1}}^{1,1} \subset C_{1} \times \lbrace t_{1} \rbrace$ and $B_{t_{3}}^{1,1} \subset C_{1} \times \lbrace t_{3} \rbrace$; 
\item $B_{t_{2}}^{2,1} \subset C_{2} \times \lbrace t_{2} \rbrace$ and $B_{t_{3}}^{2,1} \subset C_{1} \times \lbrace t_{3} \rbrace$; 
\item $B_{t_{1}}^{3,1} \subset C_{1} \times \lbrace t_{1} \rbrace$ and $B_{t_{2}}^{3,1} \subset C_{2} \times \lbrace t_{2} \rbrace$;
\item for each $t_{i} < t < t_{i+1}$, $S \times \lbrace t \rbrace$ is transverse to $B_{t}^{k,1}$. 
\end{itemize}
\item We define $3b$ tubes $\lbrace T_{1,1} ... T_{1,b} \rbrace$, $\lbrace T_{2,1} ... T_{2,b} \rbrace$, 
$\lbrace T_{3,1} ... T_{3,b} \rbrace$ in the same way, for each of the $b$ components of $\partial S$. We set $X_{k} = \big( X_{k}' \cup ( \cup_{l \in [1,b]} T_{k,l}) \big) \setminus Int (\cup_{i \neq k} \cup_{l \in [1,b]} T_{i,l}) $.
\end{itemize}
Informally, the $3$--balls constituting each tube begin their path in $\lbrace t_{i} \rbrace$ fully included in $C_{1} \times \lbrace t_{i} \rbrace$ or $C_{2} \times \lbrace t_{i} \rbrace$, then move tranversely to $S$ to the position occupied in $\lbrace t_{i+1} \rbrace$ (see Figure \ref{fig:compressionex}). All this allows us to get a compact surface for $\big( X_{i} \cap X_{j} \big) \cap \partial X$.  
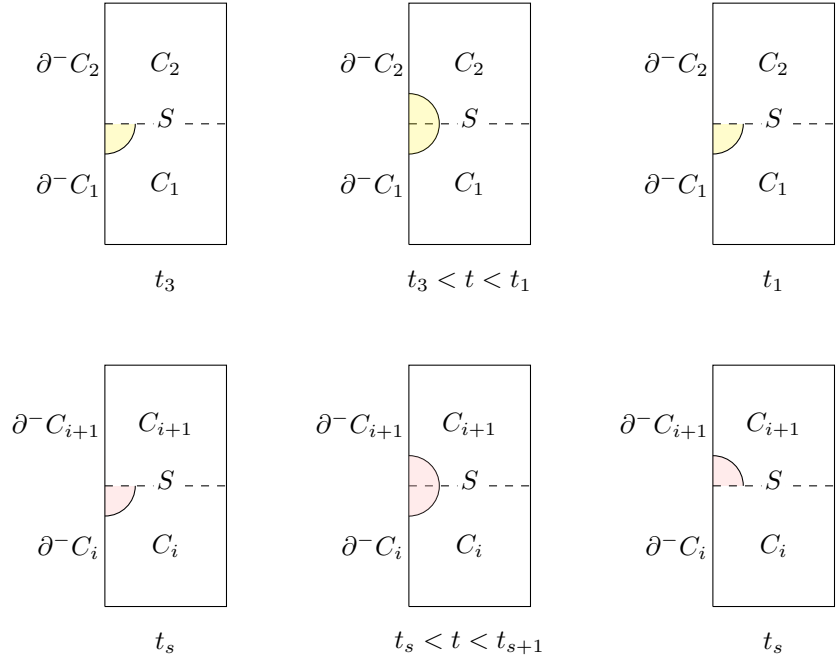
\begin{figure}[h!]
\[
\begin{tikzpicture} [scale=0.8]
\fill [yellow, opacity=0.2] (0,1.5) arc (-90:90:0.5cm and 0.5cm);
\fill [white] (0,2) -- (1,2) -- (1,3) -- (0,3) -- (0,2) -- cycle ;
\draw (0,0) -- (2,0) -- (2,4) -- (0,4) -- (0,0);
\draw [dashed] (0,2) -- (2,2);
\fill [white] (1,2) circle (0.2cm);
\node (a) at (1,1) {$C_{1}$};
\node (a) at (1,3) {$C_{2}$};
\node (a) at (1,2.1) {$S$};
\node (a) at (-0.6,3) {$\partial^{-} C_{2}$};
\node (a) at (-0.6,1) {$\partial^{-} C_{1}$};
\node (a) at (1,-0.6) {$t_{3}$};
\draw (0,1.5) arc (-90:0:0.5cm and 0.5cm);

\begin{scope}[shift={(5,0)}]
\draw (0,0) -- (2,0) -- (2,4) -- (0,4) -- (0,0);
\draw [dashed] (0,2) -- (2,2);
\fill [white] (1,2) circle (0.2cm);
\node (a) at (1,1) {$C_{1}$};
\node (a) at (1,3) {$C_{2}$};
\node (a) at (1,2.1) {$S$};
\node (a) at (-0.6,3) {$\partial^{-} C_{2}$};
\node (a) at (-0.6,1) {$\partial^{-} C_{1}$};
\node (a) at (1,-0.6) {$t_{3} < t < t_{1}$};
\draw (0,1.5) arc (-90:90:0.5cm and 0.5cm);
\fill [yellow, opacity=0.2] (0,1.5) arc (-90:90:0.5cm and 0.5cm);
\end{scope}

\begin{scope}[shift={(10,0)}]
\fill [yellow, opacity=0.2] (0,1.5) arc (-90:90:0.5cm and 0.5cm);
\fill [white] (0,2) -- (1,2) -- (1,3) -- (0,3) -- (0,2) -- cycle ;
\draw (0,0) -- (2,0) -- (2,4) -- (0,4) -- (0,0);
\draw [dashed] (0,2) -- (2,2);
\fill [white] (1,2) circle (0.2cm);
\node (a) at (1,1) {$C_{1}$};
\node (a) at (1,3) {$C_{2}$};
\node (a) at (1,2.1) {$S$};
\node (a) at (-0.6,3) {$\partial^{-} C_{2}$};
\node (a) at (-0.6,1) {$\partial^{-} C_{1}$};
\node (a) at (1,-0.6) {$t_{1}$};
\draw (0,1.5) arc (-90:0:0.5cm and 0.5cm);
\end{scope}

\begin{scope}[shift={(5,-6)}]
\draw (0,0) -- (2,0) -- (2,4) -- (0,4) -- (0,0);
\draw [dashed] (0,2) -- (2,2);
\fill [white] (1,2) circle (0.2cm);
\node (a) at (1,1) {$C_{i}$};
\node (a) at (1,3) {$C_{i+1}$};
\node (a) at (1,2.1) {$S$};
\node (a) at (-0.8,3) {$\partial^{-} C_{i+1}$};
\node (a) at (-0.6,1) {$\partial^{-} C_{i}$};
\node (a) at (1,-0.6) {$t_{s} < t < t_{s+1}$};
\draw (0,1.5) arc (-90:90:0.5cm and 0.5cm);
\fill [pink, opacity=0.3] (0,1.5) arc (-90:90:0.5cm and 0.5cm);
\end{scope}

\begin{scope}[shift={(0,-6)}]
\fill [pink, opacity=0.3] (0,1.5) arc (-90:90:0.5cm and 0.5cm);
\fill [white] (0,2) -- (1,2) -- (1,3) -- (0,3) -- (0,2) -- cycle ;
\draw (0,0) -- (2,0) -- (2,4) -- (0,4) -- (0,0);
\draw [dashed] (0,2) -- (2,2);
\fill [white] (1,2) circle (0.2cm);
\node (a) at (1,1) {$C_{i}$};
\node (a) at (1,3) {$C_{i+1}$};
\node (a) at (1,2.1) {$S$};
\node (a) at (-0.8,3) {$\partial^{-} C_{i+1}$};
\node (a) at (-0.6,1) {$\partial^{-} C_{i}$};
\node (a) at (1,-0.6) {$t_{s}$};
\draw (0,1.5) arc (-90:0:0.5cm and 0.5cm);
\end{scope}

\begin{scope}[shift={(10,-6)}]
\fill [pink, opacity=0.3] (0,1.5) arc (-90:90:0.5cm and 0.5cm);
\fill [white] (0,2) -- (1,2) -- (1,1) -- (0,1) -- (0,2) -- cycle ;
\draw (0,0) -- (2,0) -- (2,4) -- (0,4) -- (0,0);
\draw [dashed] (0,2) -- (2,2);
\fill [white] (1,2) circle (0.2cm);
\node (a) at (1,1) {$C_{i}$};
\node (a) at (1,3) {$C_{i+1}$};
\node (a) at (1,2.1) {$S$};
\node (a) at (-0.8,3) {$\partial^{-} C_{i+1}$};
\node (a) at (-0.6,1) {$\partial^{-} C_{i}$};
\node (a) at (1,-0.6) {$t_{s}$};
\draw (0.5,2) arc (0:90:0.5cm and 0.5cm);
\end{scope}

\end{tikzpicture}
\]
\caption{Visualising the positions of the $3$--balls on the compression bodies\\{\footnotesize  With a yellow $B_{1,l} $ and a pink $B_{2,l}$ or $B_{3,l}$}} 
\label{fig:compressionex}
\end{figure}
\begin{Remark}
The number of tubes allows the number of boundary components of the pages to match that of the central surface (and the number of boundary components of $\partial X$).
\end{Remark}
\begin{Remark}
The $3$--balls do not affect, up to diffeomorphism, the $C_{i}$'s, nor the surfaces $S$ and $\partial_{-}C_{i}$. Therefore, we will simply refer to a $C_{i}$ minus the interior of the family of $3$--balls defined above as a $C_{i}$, and to $(S \times \lbrace t \rbrace) \setminus Int \big( (S \times \lbrace t \rbrace) \cap \cup_{k,l} B_{t}^{k,l} \big) $ as $( S \times \lbrace t \rbrace)$.  
\end{Remark}
\begin{Proposition}
\label{prop:triscase1}
The decomposition $X=X_{1} \cup X_{2} \cup X_{3}$ is a relative trisection of $X$.
\end{Proposition}
\begin{proof} We just need to check that the pieces of the decomposition corresponds to those defined in \ref{def:relative}.
\begin{itemize}
\item Each $X_{k}$ is a union of a $4$--dimensional handlebody $C_{i} \times I$ and $b$ $4$--dimensional $1$--handles (our tubes  $\cup_{l \in [1,b]} T_{k,l}$). Therefore each $X_{k}$ is a $4$--dimensional handle\-body.
\item The triple intersection $X_{1} \cap X_{2} \cap X_{3}$ is composed of three copies of $S$ (the $S \times \lbrace t_{i} \rbrace$ for $i=1,2,3$), whose boundary components are joined by $3b$ $2$--dimensional $1$--handles $\big( B_{t}^{k,l} \cap (S \times \lbrace t \rbrace) \big) \times I$, as shown on Figure \ref{triplinterpres} and Figure \ref{triplinterflips}. So $X_{1} \cap X_{2} \cap X_{3}$ is a compact connected surface with non empty boundary.
\item The intersection $(X_{1} \cap X_{2}) \cap \partial X$ is a boundary connected sum of $\partial_{-}C_{1} \times \lbrace t_{3} \rbrace$ and bands (see Figure \ref{doublintercas1} and Figure \ref{doublinterbordcas1}). Therefore it is also a compact surface. The other intersections $(X_{2} \cap X_{3}) \cap \partial X$ and $(X_{3} \cap X_{1}) \cap \partial X$ are built in the same way: all three surfaces are diffeomorphic. Note that they can be disconnected if $\partial M$ is disconnected. In this case, they have the same number of components as $\partial X$ (if the monodromy glues together two different boundary components of $M$, it will act on the pages accordingly). 
\item  As $X_{1} \cap X_{2}$ is composed of: 
\begin{itemize}
\item $C_{1} \times \lbrace t_{3} \rbrace$;
\item $S \times [t_{1},t_{2}]$ (a tridimensional handlebody, since $S$ has boundary);
\item $2b$ tridimensional $1$--handles linking the above handlebodies;
\end{itemize}
it is a tridimensional handlebody, that we will call $C_{1,2}$ (see Figure \ref{doublintercas1}). Moreover, Figure \ref{doublinterbordcas1} and Figure \ref{doublintertriplecas1} depict the decomposition of the boundary of $X_{1} \cap X_{2}$ as:
\begin{equation*}
\partial (X_{1} \cap X_{2}) = \big( X_{1} \cap X_{2} \cap X_{3} \big) \cup \big( (X_{1} \cap X_{2}) \cap \partial X \big)
\end{equation*}  
This gives $C_{1,2}$ a structure of compression body, with $\partial_{+} C_{1,2} = X_{1} \cap X_{2} \cap X_{3}$ and $\partial_{-} C_{1,2} = (X_{1} \cap X_{2}) \cap \partial X$.

The intersections $X_{2} \cap X_{3}$ and $X_{3} \cap X_{1}$ are built in the same way.
\item Finally, we have:
\begin{equation*}
X_{1} \cup \partial X \simeq (\partial_{-} C_{1} \times [t_{1},t_{3}]) \cup \big( (\cup_{l \in [1,b]} T_{1,l}) \cup \partial X \big) \simeq \big( (X_{1} \cap X_{2})\cap \partial X \big) \times I
\end{equation*}
and $X_{2} \cap \partial X$ or $X_{3} \cap \partial X$ are built in the same way.
\end{itemize}
Therefore $X=X_{1} \cup X_{2} \cup X_{3}$ is a relative trisection of $X$. 
\end{proof}
\begin{figure}[h!]
\[
\begin{tikzpicture}[scale=0.7]
\node (mypic) at (0,0) {\includegraphics[scale=0.5]{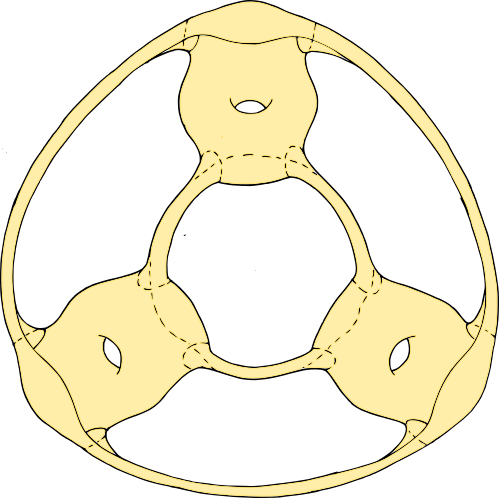}};
\node[circle, draw, fill=white] (a) at (4.7,-3.5) {$t_{3}$};
\node[circle, draw, fill=white] (a) at (0,5.5) {$t_{2}$};
\node[circle, draw, fill=white] (a) at (-4.7,-3.5) {$t_{1}$};
\node [ellipse, draw, fill=white](a) at (0.5,-5.2) {$0 \sim 1$};
\end{tikzpicture}
\]
\caption{A view of $X_{1} \cap X_{2} \cap X_{3}$, Case 1, $\phi$ preserves the components of $\partial S$}
\label{triplinterpres}
\end{figure}
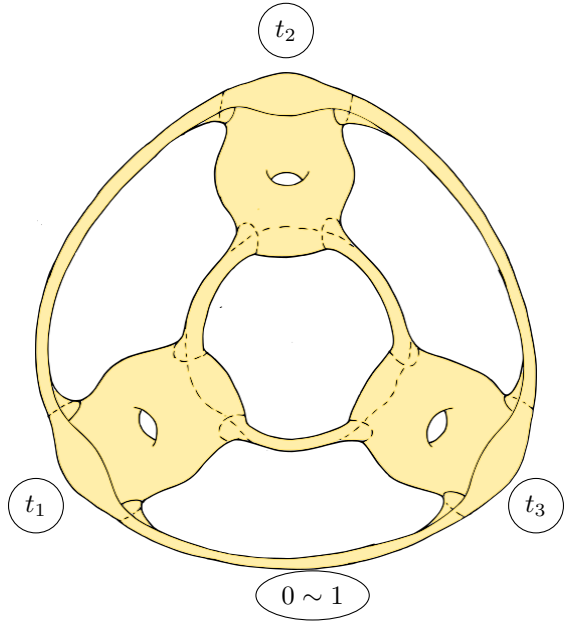

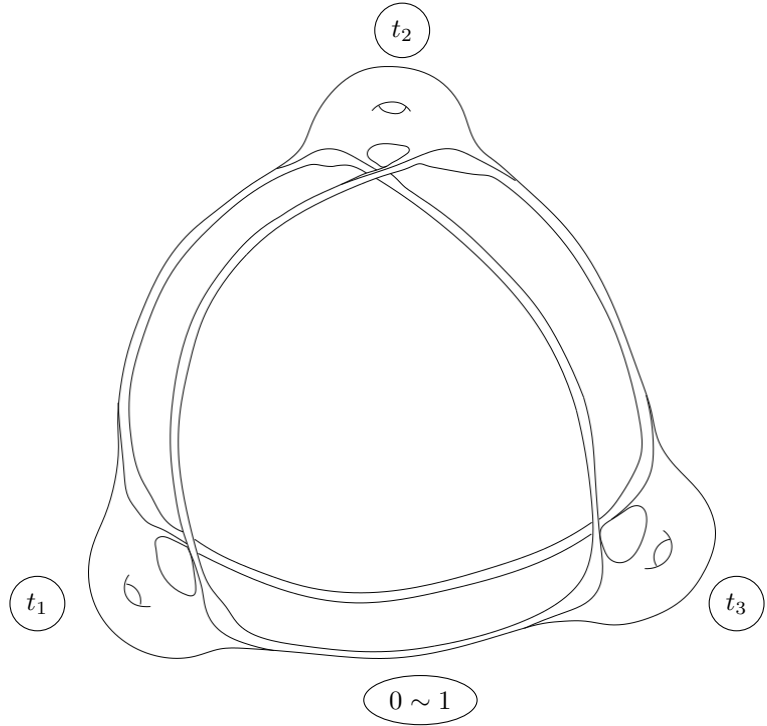
\begin{figure}[h!]
\[
\begin{tikzpicture}[scale=0.8]
\node (mypic) at (0,0) {\includegraphics[scale=0.5]{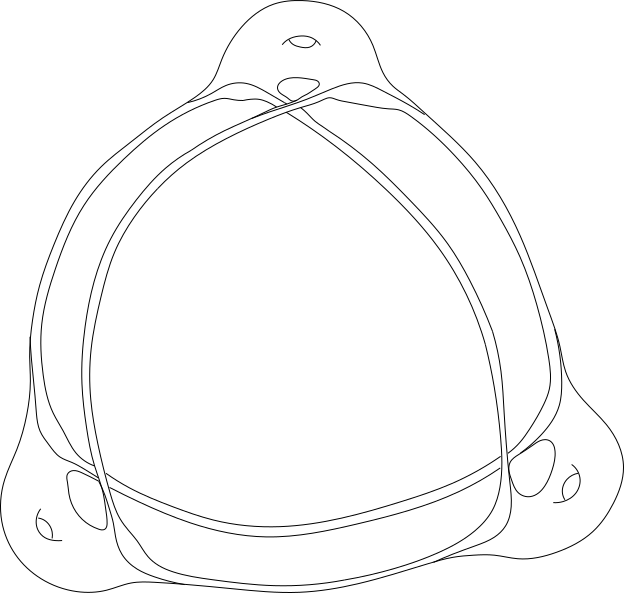}};
\node[circle, draw, fill=white] (a) at (5.5,-4) {$t_{3}$};
\node[circle, draw, fill=white] (a) at (0,5.5) {$t_{2}$};
\node[circle, draw, fill=white] (a) at (-6,-4) {$t_{1}$};
\node [ellipse, draw, fill=white](a) at (0.3,-5.6) {$0 \sim 1$};

\end{tikzpicture}
\]
\caption{A view of $X_{1} \cap X_{2} \cap X_{3}$, Case 1, when $\phi$ flips the components of $\partial S$}
\label{triplinterflips}
\end{figure}

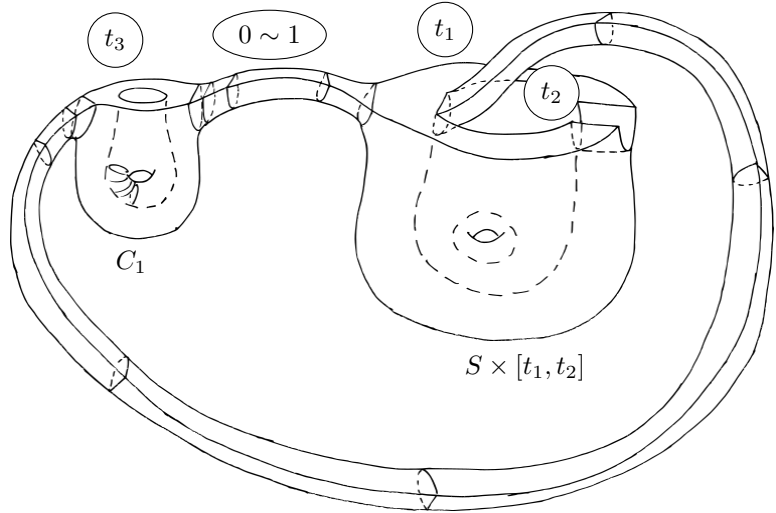
\begin{figure}[h!]
\[
\begin{tikzpicture}[scale=0.7]
\node (mypic) at (0,0) {\includegraphics[scale=0.5]{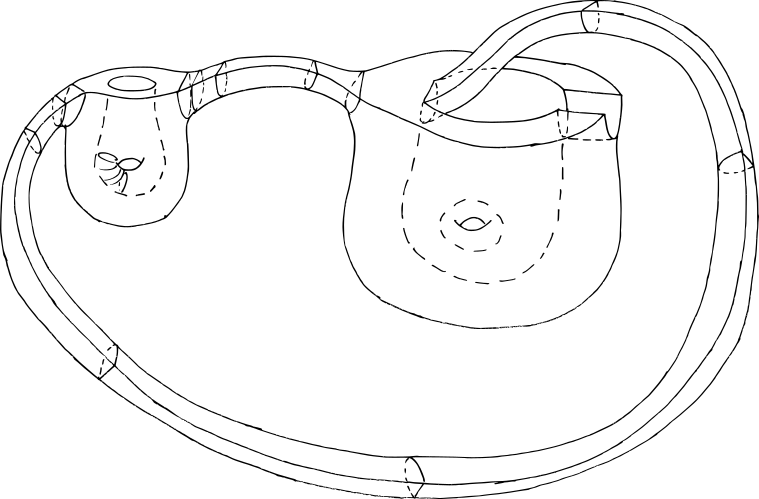}};
\node[circle, draw, fill=white] (a) at (-5.2,4.2) {$t_{3}$};
\node[circle, draw, fill=white] (a) at (1,4.4) {$t_{1}$};
\node[circle, draw, fill=white] (a) at (3,3.2) {$t_{2}$};
\node [ellipse, draw, fill=white](a) at (-2.3,4.3) {$0 \sim 1$};
\node (b) at (-4.9,0) {$C_{1}$};
\node (b) at (2.5,-2) {$S \times [t_{1},t_{2}]$};
\end{tikzpicture}
\]
\caption{A view of $X_{1} \cap X_{2}$, case 1}
\label{doublintercas1}
\end{figure}

\begin{figure}[h!]
\[
\begin{tikzpicture}[scale=0.8]
\node (mypic) at (0,0) {\includegraphics[scale=0.64]{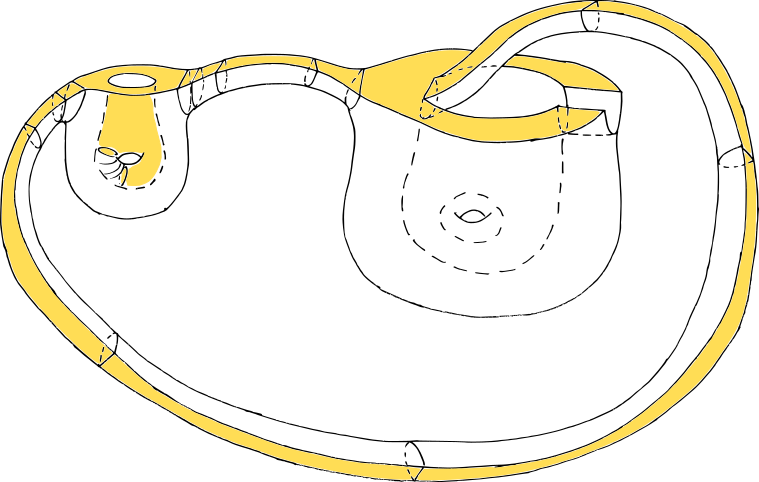}};
\node[circle, draw, fill=white] (a) at (-5.2,4.4) {$t_{3}$};
\node[circle, draw, fill=white] (a) at (1,4.5) {$t_{1}$};
\node[circle, draw, fill=white] (a) at (3.2,3.2) {$t_{2}$};
\node [ellipse, draw, fill=white](a) at (-2.5,4.5) {$0 \sim 1$};
\node (b) at (-5.1,0) {$C_{1}$};
\node (b) at (2.5,-2) {$S \times [t_{1},t_{2}]$};
\end{tikzpicture}
\]
\caption{The page $(X_{1} \cap X_{2}) \cap \partial X$, coloured in yellow on $X_{1} \cap X_{2}$, Case 1}
\label{doublinterbordcas1}
\end{figure}
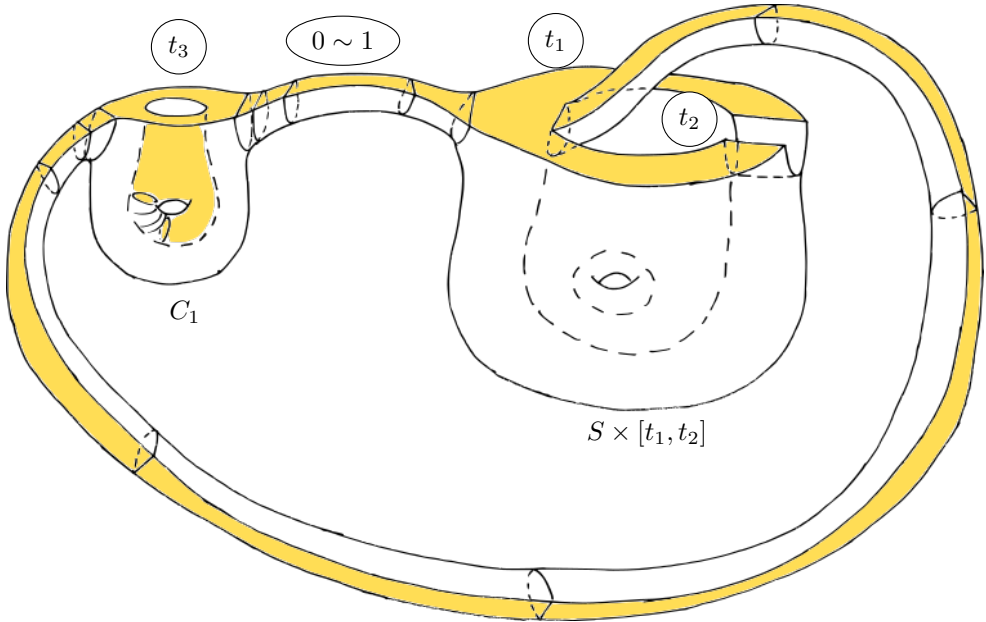

\begin{figure}[h!]
\[
\begin{tikzpicture}[scale=0.8]
\node (mypic) at (0,0) {\includegraphics[scale=0.64]{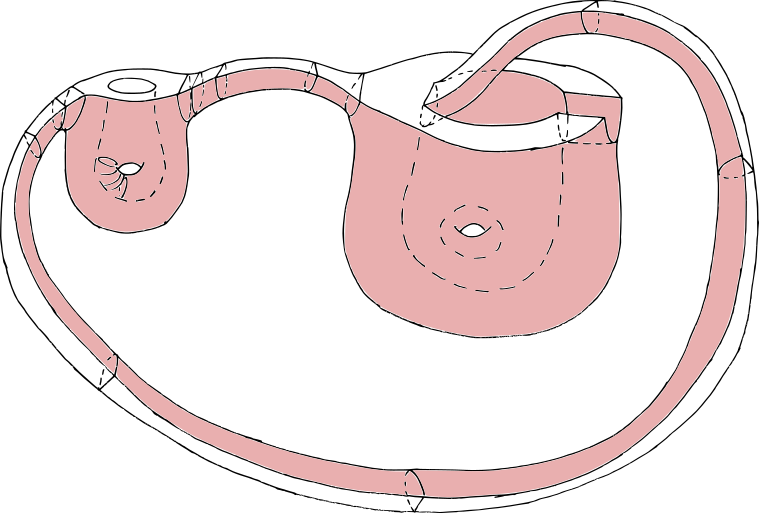}};
\node[circle, draw, fill=white] (a) at (-5,4.6) {$t_{3}$};
\node[circle, draw, fill=white] (a) at (1,4.5) {$t_{1}$};
\node[circle, draw, fill=white] (a) at (2.9,3.5) {$t_{2}$};
\node [ellipse, draw, fill=white](a) at (-2.5,4.7) {$0 \sim 1$};
\node (b) at (-5.1,0) {$C_{1}$};
\node (b) at (2.5,-2.1) {$S \times [t_{1},t_{2}]$};
\end{tikzpicture}
\]
\caption{The central surface coloured in pink on $X_{1} \cap X_{2}$, Case 1}
\label{doublintertriplecas1}
\end{figure}
We can now prove the following result (a more precise but somewhat more engaging version of Theorem \ref{th:1flip}).
\begin{Theorem}
\label{th2}
Let $M$ be a smooth, compact, oriented, connected $3$--manifold and $\phi$ an orientation-preserving self-diffeomorphism of $M$. Let $X$ be the smooth oriented bundle over the circle with fiber $M$ and monodromy $\phi$. The monodromy acts by permutation on the connected components $\lbrace \partial^{1} M, ... , \partial^{\ell_{M}}M \rbrace$ of $\partial M$, and we call $\sigma_{\phi, \partial M }$ this action. Suppose that $M$ admits a sutured Heegaard splitting $M=C_{1} \cup_{S} C_{2}$ that is flipped by $\phi$. We set the following notations:
\begin{itemize}
\item the surface $S$ is of genus $g$ with $b$ boundary components;
\item we write $p$ the sum of the genera of the components of $\partial_{-}C_{1} \simeq \partial_{-} C_{2}$;
\item the orbits of $\sigma_{\phi, \partial M }$ are $\lbrace \mathcal{O}_{1}, ... , \mathcal{O}_{\ell_{X}} \rbrace$, where $\ell_{X}$ is the number of boundary components of $X$; we write $p_{k}$ the sum of the genera of the components of $\partial_{-}C_{1}$ (or equivalently of $\partial_{-} C_{2}$) within $\mathcal{O}_{k}$ and $b_{k}$ the sum of their number of boundary components; we write $c_{\mathcal{O}_{k}}$ the number of components of $\partial_{-}C_{1}$ (or equivalently of $\partial_{-} C_{2}$) in $\mathcal{O}_{k}$;
\item as $\phi$ also acts by permutation on the components of $\partial S$ within each orbit, we write $c_{\phi}^{k}$ the number of induced orbits within a given $\mathcal{O}_{k}$, and $c_{\phi} = \sum_{k=1}^{\ell_{X}} c_{\phi}^k$ the total number of orbits of the action of $\phi$ on the components of $\partial S$.
\end{itemize} 
Then $X$ admits a $(G,k;P,B)$--relative trisection, where:
\begin{itemize}
\item $G=3g+3b-c_{\phi}-2$;
\item $k=g+p+2b-1$;
\item $P=(P_{k})_{1 \leq k \leq \ell_{X}}$, with $P_{k} = p_{k} + b_{k} - c_{\mathcal{O}_{k}} - c_{\phi}^{k} +1$;
\item $B=(B_{k})_{1 \leq k \leq \ell_{X}}$, with $B_{k}=2 c_{\phi}^{k}$.
\end{itemize}
The Euler characteristic of the central surface is given by $\chi = 3(\chi(S)-b)$.
\end{Theorem}

\begin{proof}
We only need to compute the parameters of the trisection defined in Proposition \ref{prop:triscase1}. Each $X_{i}$ is a union of a $C_{k} \times I$ (a $4$--dimensional handlebody of genus $g+p+b-1$) and $b$ $4$--dimensional $1$--handles. Therefore $k=g+p+2b-1$. To compute the genus and number of boundary components of the triple intersection and of the pages, we use their Euler characteristic $\chi$. Recall that gluing a band by two opposite edges on a triangulated surface $S$ produces a surface $S'$ with $\chi(S')=\chi (S) -1$. As the central surface consists of three copies of $S$ joined by $3b$ bands, we obtain that its Euler characteristic is $3(\chi(S) - b)$. Using the fact that we also have $\chi (F_{G,B}) = 2-2G-B$, we can derive $G$ from $\chi$ just by counting the boundary components of the central surface. The monodromy acts on the boundary components of $S$ by permutation. One orbit of this permutation creates exactly two boundary components of the central surface. So if we denote by $c_{\phi}$ the number of orbits of the permutation, we obtain that $B=2 c_{\phi}$, and that $G=3g+3b-c_{\phi}-2$. The same applies to the genus $P_{k}$ of each component of a page. One component is constituted of $c_{\mathcal{O}_{k}}$ components of a $\partial_{-} C_{i}$ that belong to the same orbit, joined by $b_{k}$ bands, the resulting surface having $B_{k} = 2 c_{\phi}^{k}$ boundary components. We get that $P_{k}=p_{k}+b_{k}-c_{\mathcal{O}_{k}} - c_{\phi}^{k} + 1$. 
\end{proof}
\begin{Example}
If the monodromy preserves all the boundary components of $S$ and if the boundary of $M$ is connected, we obtain a $(3g+2b-2,g+p+2b-1;p,2b)$--relative trisection of $X$. 
\end{Example}

\subsection{Case 2: $\phi$ preserves a sutured Heegaard splitting of $M$}
The process and notations are the same as in the previous case. We use a decomposition as shown on Figure \ref{trisXcas2}. We will not reiterate the proof that this decomposition is indeed a relative trisection of $X$, as it is essentially the same as the proof of Proposition \ref{prop:triscase1}. 

We just outline a few facts. Now the negative boundaries of the compression bo\-dies are not necessarily diffeomorphic. Because we do need the pages $(X_{i} \cap X_{j}) \cap \partial X$ to be diffeomorphic, we split the interval $I$ into six subintervals. Therefore the genus of the relative trisection will be higher than in Case 1, as the central surface will be constituted of six copies of $S$ joined by bands. Note that the monodromy can still flip some boundary components of $S$ or $M$.

\begin{figure}[h!]
\[
\begin{tikzpicture}
\fill [red!50] (0,0) -- (3,0) -- (3,2) -- (0,2) -- (0,0);
\fill [blue!50] (3,0) -- (5,0) -- (5,2) -- (3,2) -- (3,0);
\fill [blue!50] (6,2) -- (8,2) -- (8,4) -- (6,4) -- (6,2);
\fill [red!50] (4,2) -- (6,2) -- (6,4) -- (4,4) -- (4,2);
\fill [green!50] (2,2) -- (4,2) -- (4,4) -- (2,4) -- (2,2);
\fill [blue!50] (0,2) -- (2,2) -- (2,4) -- (0,4) -- (0,2);
\fill [green!50] (5,0) -- (7,0) -- (7,2) -- (5,2) -- (5,0);
\fill [red!50] (7,0) -- (8,0) -- (8,2) -- (7,2) -- (7,0);
\draw (0,0) -- (8,0) -- (8,4) -- (0,4) -- (0,0);
\draw [dashed] (0,0) -- (0,-0.5);
\draw [dashed] (8,0) -- (8,-0.5); 
\node (a) at (0,4.2) {0};
\node (a) at (8,4.2) {1};
\node (a) at (0,-0.7) {$\phi$};
\node (a) at (8,-0.7) {$\phi$};
\fill [green!50] (4,1.9) -- (5,1.9) -- (5,2.1) -- (4,2.1) -- (4,1.9);
\fill [blue!50] (5,1.9) -- (6,1.9) -- (6,2.1) -- (5,2.1) -- (5,1.9);
\fill [green!50] (7,1.9) -- (8,1.9) -- (8,2.1) -- (7,2.1) -- (7,1.9);
\fill [blue!50] (2,1.9) -- (3,1.9) -- (3,2.1) -- (2,2.1) -- (2,1.9);
\fill [green!50] (0,1.9) -- (2,1.9) -- (2,2.1) -- (0,2.1) -- (0,1.9);
\fill [red!50] (3,1.9) -- (4,1.9) -- (4,2.1) -- (3,2.1) -- (3,1.9);
\fill [red!50] (6,1.9) -- (7,1.9) -- (7,2.1) -- (6,2.1) -- (6,1.9);
\draw (2,1.9) -- (2,4);
\draw (4,1.9) -- (4,4);
\draw (6,1.9) -- (6,4);
\draw (3,2.1) -- (3,0);
\draw (5,2.1) -- (5,0);
\draw (7,2.1) -- (7,0);
\draw (0,1.9) -- (8,1.9) -- (8,2.1) -- (0,2.1) -- (0,2.1);
\draw [stealth - stealth] (8.5,4) -- (8.5,2.5);
\draw [stealth - stealth] (8.5,0) -- (8.5,1.5); 
\node (b) at (9,2) {$S \simeq F_{g,b}$};
\node (b) at (8.8,1) {$C_{1}$};
\node (b) at (8.8,3) {$C_{2}$};
\node (a) at (3,4.2) {$t_{2}$};
\node (a) at (2,4.2) {$t_{1}$};
\node (a) at (4,4.2) {$t_{3}$};
\node (a) at (5,4.2) {$t_{4}$};
\node (a) at (6,4.2) {$t_{5}$};
\node (a) at (7,4.2) {$t_{6}$};
\node (c) at (0.4,3.5) {$X_{1}$};
\node (c) at (7.6,0.5) {$X_{2}$};
\node (c) at (7.6,3.5) {$X_{1}$};
\node (c) at (0.4,0.5) {$X_{2}$};
\node (c) at (3,3.5) {$X_{3}$};
\node (c) at (6,0.5) {$X_{3}$};
\node (c) at (4,0.5) {$X_{1}$};
\node (c) at (5,3.5) {$X_{2}$};
\draw [red] (3,2) circle (0.1cm);
\draw [red] (2,2) circle (0.1cm);
\draw [red] (4,2) circle (0.1cm);
\draw [red] (5,2) circle (0.1cm);
\draw [red] (6,2) circle (0.1cm);
\draw [red] (7,2) circle (0.1cm);
\end{tikzpicture}
\]
\caption{A relative trisection of $X$, Case 2}
\label{trisXcas2}
\end{figure}
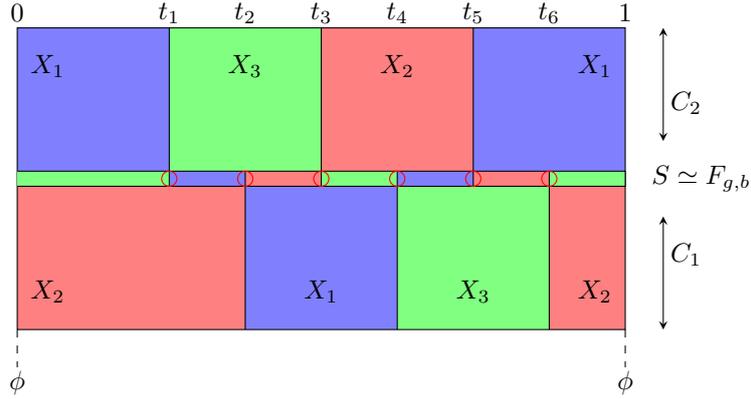

\begin{Theorem}
\label{th1}
Let $M$ be a smooth, compact, oriented, connected $3$--manifold and $\phi$ an orientation-preserving self-diffeomorphism of $M$. Let $X$ be the smooth oriented bundle over the circle with fiber $M$ and monodromy $\phi$. The monodromy acts by permutation on the connected components $\lbrace \partial^{1} M, ... , \partial^{\ell_{M}}M \rbrace$ of $\partial M$, and we call $\sigma_{\phi, \partial M }$ this action. Let $M=C_{1} \cup_{S} C_{2}$ be a sutured Heegaard splitting of $M$ that is preserved by $\phi$. We set the following notations: 
\begin{itemize}
\item the surface $S$ is of genus $g$ with $b$ boundary components;
\item we write $p_{i}$ the sum of the genera of the components of $\partial_{-}C_{i}$;
\item we denote the orbits of $\sigma_{\phi, \partial M }$ by $\lbrace \mathcal{O}_{1}, ... , \mathcal{O}_{\ell_{X}} \rbrace$, where $\ell_{X}$ is the number of boundary components of $X$; we write $p_{k,i}$ the sum of the genera of the components of $\partial_{-}C_{i}$ within $\mathcal{O}_{k}$ and $b_{k}$ the sum of their numbers of boundary components; we write $c_{\mathcal{O}_{k},i}$ the number of components of $\partial_{-}C_{i}$ in $\mathcal{O}_{k}$;
\item as $\phi$ also acts by permutation on the components of $\partial S$ within each orbit, we write $c_{\phi}^{k}$ the number of induced orbits within a given $\mathcal{O}_{k}$, and $c_{\phi} = \sum_{k=1}^{\ell_{X}} c_{\phi}^k$ the total number of orbits of the action of $\phi$ on the components of $\partial S$.
\end{itemize} 
Then $X$ admits a $(G,k;P,B)$--relative trisection, where:
\begin{itemize}
\item $G=6g+6b-c_{\phi}-5$;
\item $k=2g+p_{1} + p_{2} + 4b -3$;
\item $P=(P_{k})_{1 \leq k \leq \ell_{X}}$, with $P_{k} = p_{k,1} + p_{k,2} + 2b_{k} - c_{\mathcal{O}_{k},1}- c_{\mathcal{O}_{k},2} - c_{\phi}^{k} +1$;
\item $B=(B_{k})_{1 \leq k \leq \ell_{X}}$, with $B_{k}=2 c_{\phi}^{k}$.
\end{itemize}
The Euler characteristic of the central surface is given by $\chi = 6(\chi(S)-b)$.
\end{Theorem}
\begin{proof}
We just give a sketch of proof, as the arguments are essentially the same as those produced for Theorem \ref{th2}. We can still compute the genus of this relative trisection using the Euler characteristic of the central surface. We obtain a central surface of Euler characteristic $6(\chi (S) -b)$, where $S \simeq F_{g,b}$ is the Heegaard surface of the splitting. Then we use the Euler characteristics of the components of the pages and the number of boundary components of the surfaces involved to compute the parameters featured in Theorem \ref{th1}. 
\end{proof}
\begin{Example}
For instance, if the monodromy preserves all the boundary components of $S$ and if the boundary of $M$ is connected, we obtain a $(6g+5b-5,2g+ p_{1} + p_{2} + 4b-3; p_{1} + p_{2} +b-1,2b)$--relative trisection of $X$.
\end{Example}
\begin{Remark} 
In \cite{koenig2017trisections}, Koenig builds trisections of the same (optimal) genus in both cases. His construction is less straightforward if the monodromy preserves a Heegaard splitting of $M$ and the question of its adaptability to the compact setting is open.
\end{Remark}

\section{The diagrams}
Recall that $X$ is a fiber bundle over $S^{1}$, with fiber $M=C_{1} \cup_{S} C_{2}$. Using a sutured Heegaard diagram of the fiber, we draw the diagrams corresponding to the relative trisections of Section 4, also considering the two separate cases.
\subsection{A specific compression body}
\label{productcompbody}
While constructing our relative trisections, we encountered a certain type of compression body, that we denote by $C_{g,b} = F_{g,b} \times I$, with:
\begin{itemize}
\item its positive boundary $\partial_{+} C_{g,b}$ the union of $F_{g,b} \times \lbrace 0 \rbrace$ and $F_{g,b} \times \lbrace 1 \rbrace$, with matching  boundary components joined two by two by bands;
\item its negative boundary a disjoint union of $b$ disks, each one capping off a bounda\-ry component of $\partial_{+} C_{g,b}$.
\end{itemize}
We want to define a cut system on $\partial_{+}C_{g,b}$ corresponding to this compression body. If an arc $a$ is properly embbedded in $F_{g,b}$, then $a \times I$ is a properly embedded disk in $C_{g,b}$, with boundary the curve $(a \times \lbrace 0 \rbrace ) \cup (\partial a \times I) \cup (-a \times \lbrace 1 \rbrace )$. By placing $\partial a$ on a band joining $F_{g,b} \times \lbrace 0 \rbrace$ to $F_{g,b} \times \lbrace 1 \rbrace$, we ensure that the boundary of the disk lies on $\partial_{+} C_{g,b}$. Therefore such a curve can be a candidate for our cut system. 
\begin{Remark}
\label{rmkembed}
To draw our diagrams in the next figures, we choose an embedding of $\partial_{+} C_{g,b}$ with $F_{g,b} \times \lbrace 0 \rbrace$ symmetric to $F_{g,b} \times \lbrace 1 \rbrace$: with this embedding, the curve on the previous discussion will be constituted of an arc on $F_{g,b} \times \lbrace 0 \rbrace$, connected to its symmetric on $F_{g,b} \times \lbrace 1 \rbrace$ by disjoint arcs drawn on the relevant band (or, informally, any curve on $\partial_{+}C_{g,b}$ looking symmetric in $F_{g,b} \times \lbrace 0 \rbrace$ and $F_{g,b} \times \lbrace 1 \rbrace$ bounds a properly embedded disk in $C_{g,b}$).
\end{Remark} 
We define the following cut system, represented on Figure \ref{specificcompbdycs}: 
\begin{itemize}
\item From $2g$ non isotopic and properly embedded arcs on $F_{g,b} \times \lbrace 0 \rbrace$, that cut the surface into a disk with $(b-1)$ holes, we obtain $2g$ non isotopic curves on $\partial_{+} C_{g,b}$ defined as above (in deep blue on Figure \ref{specificcompbdycs}). After performing surgery along these curves, we are down to a surface of genus $(b-1)$ with $b$ boundary components. 
\item The previous step singled out a boundary component of $F_{g,b} \times \lbrace 0 \rbrace$ (in dotted light blue on Figure \ref{specificcompbdycs}), whose glued on band was used to draw the $(\partial a \times I)$ part of the curves. We choose $(b-1)$ arcs linking this component to the others, then draw the $(b-1)$ curves obtained as above from these arcs (symmetric curves in green on Figure \ref{specificcompbdycs}). After performing surgery along these curves, we obtain a sphere with $b$ boundary components.
\item We draw $(b-1)$ curves, each one parallel to a boundary component (in red on Figure \ref{specificcompbdycs}); surgery along these curves produces the $b$ disjoint disks corres\-ponding to $\partial_{-} C_{g,b}$. 
\item As all these curves bound properly embedded disks in $C_{g,b}$, we are done. 
\end{itemize} 

\begin{figure}[h!]
\[
\begin{tikzpicture}[scale=0.4]
\node (mypic) at (0,0) {\includegraphics[scale=0.5]{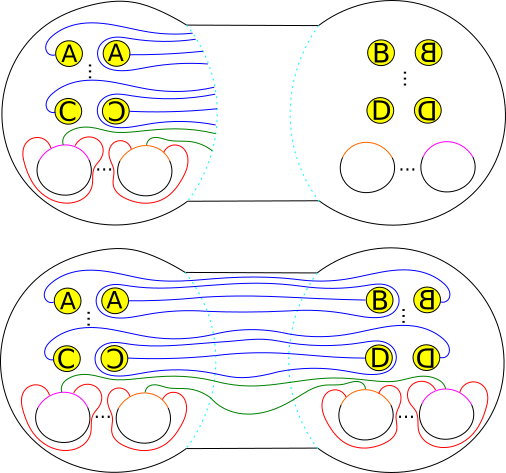}};
\end{tikzpicture}
\]
\caption{A cut system for $C_{g,b}$ \\{\footnotesize The circles bounding yellow disks are identified two by two to contribute to the genus of each copy of $F_{g,b}$; the black semi-circles represent parts of the boundary components of each copy; the coloured semi-circles (here in pink or orange) are identified by colour to form the bands; we picked one component on which we drew the connecting arcs}}
\label{specificcompbdycs}
\end{figure}

\subsection{Relative trisection diagrams, Case 1}
The monodromy flips a sutured Heegaard splitting of $M$, of which we use a sutured Heegaard diagram. Recall from Section 4 that the handlebodies $X_{i} \cap X_{j}$ are constituted of a copy of $C_{1}$ or $C_{2}$ and a handlebody $F_{g,b} \times I$, joined by $1$--handles. The central surface is constituted of three copies of $S$ joined by bands. The cut systems corresponding $C_{1}$ or $C_{2}$ can be derived from the sutured Heegaard diagram of $M$. We also need to obtain a compression body $F_{g,b} \times [t_{k},t_{k+1}] \simeq C_{g,b}$ from the last two copies of $S$ joined by bands. So we can use the results of \ref{productcompbody}, we choose the following embedding for the central surface. We label $S_{k}$ the copy $S \times \lbrace t_{k} \rbrace$. We divide the plane equally, according to three rays, $\theta_{1}$, $\theta_{2}$ and $\theta_{3}$. We draw $S_{k}$ on $\theta_{k}$, such that $S_{2}$ is the symmetric of $S_{1}$ along the ray midway between $\theta_{1}$ and $\theta_{2}$, and the symmetric of $S_{3}$ along the ray midway between $\theta_{2}$ and $\theta_{3}$ (see Figure \ref{embedsurfcase1}). The following algorithm produces a relative trisection diagram of $X$.
\begin{itemize}
\item Start with the red curves associated to $X_{1} \cap X_{2}$.
\begin{itemize}
\item At $t_{3}$, we have a copy of $C_{1}$; we draw the corresponding diagram on $S_{3}$. 
\item Between $t_{3}$ and $t_{1}$, our compression body is just a thickening of the bands joining the copies of $S$, so we do not add any curve (the same applies between $t_{2}$ and $t_{3}$) .
\item From $t_{1}$ to $t_{2}$, we have a copy of $C_{g,b}$; we draw the corresponding diagram on the union of $S_{2}$, $S_{3}$ and their linking bands, as in \ref{productcompbody}. 
\end{itemize}
\item Then we draw the blue curves corresponding to $X_{1} \cap X_{3}$ in the same fashion, remembering that this time we have a copy of $C_{1}$ at $t_{1}$ and a copy of $C_{g,b}$ from $t_{2}$ to $t_{3}$.
\item We finish with the green curves corresponding to $X_{2} \cap X_{3}$.
\begin{itemize}
\item At $t_{2}$, we have a copy of $C_{2}$; we draw the corresponding diagram on $S_{2}$. Such a diagram is symmetric to a diagram drawn on $S_{1}$ or $S_{3}$ because of the embedding of the central surface. 
\item Between $t_{2}$ and $t_{3}$, our compression body is just a thickening of the bands joining the copies of $S$, so we do not add any curve (the same applies between $t_{1}$ and $t_{2}$) .
\item From $t_{3}$ to $t_{1}$, $X_{2} \cap X_{3}$ consists in two compression bodies glued along part of their boundary, according to the monodromy $\phi$. If we consider an arc $a$ in $S_{1}$, and its image $\phi (a)$ in $S_{3}$, we can glue the two disks $(a \times [t_{1},0])$ and $(a \times [1, t_{3}])$ along $(a \times \lbrace 0 \rbrace ) \simeq (\phi (a) \times \lbrace 1 \rbrace )$. By doing so we obtain a disk in the quotient compression body, bounded by $(a \times \lbrace t_{1} \rbrace) \cup ( \partial a \times [t_{1},0]) \cup_{\phi} (\partial a \times [1,t_{3}]) \cup (- \phi (a) \times \lbrace t_{3} \rbrace)$. So if we provide a set of arcs on $S_{1}$ as in \ref{productcompbody}, we obtain, by gluing each arc to its image by $\phi$ on $S_{3}$ (with opposite direction), a cut system corresponding to the quotient compression body.  

\end{itemize}
\end{itemize}

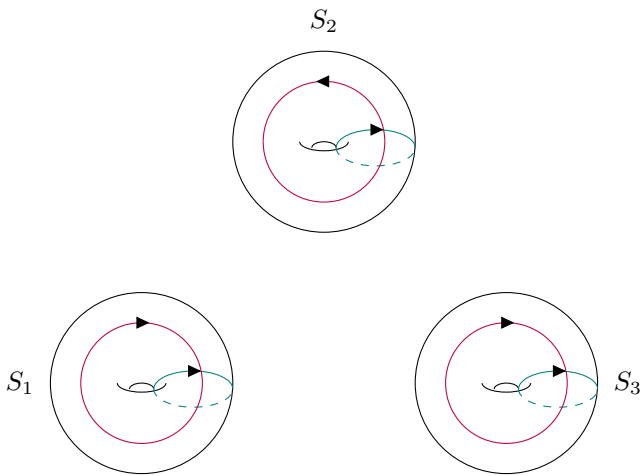
\begin{figure}[h!]
\[
\begin{tikzpicture}[scale=0.8]
\draw (0,0) circle (1.5cm);
\draw (-0.4,0) arc (180:360:0.4cm and 0.15cm);
\draw (0.2,-0.1) arc (0:180:0.2cm and 0.1cm);
\draw [dashed, teal](0.2,-0.1) arc (180:360:0.65cm and 0.3cm);
\draw [teal](0.2,-0.1) arc (180:0:0.65cm and 0.3cm)
node[
    currarrow,
    pos=0.5, 
    xscale=1,
    sloped,
    scale=1] {};
\draw [purple] (-1,0) arc (180:0:1cm)
node[
    currarrow,
    pos=0.5, 
    xscale=1,
    sloped,
    scale=1] {};
\draw [purple] (-1,0) arc (180:360:1cm);
\begin{scope}[shift={(3,4)}]
\draw (0,0) circle (1.5cm);
\draw (-0.4,0) arc (180:360:0.4cm and 0.15cm);
\draw (0.2,-0.1) arc (0:180:0.2cm and 0.1cm);
\draw [dashed, teal](0.2,-0.1) arc (180:360:0.65cm and 0.3cm);
\draw [teal](0.2,-0.1) arc (180:0:0.65cm and 0.3cm)
node[
    currarrow,
    pos=0.5, 
    xscale=1,
    sloped,
    scale=1] {};
\draw [purple] (-1,0) arc (180:0:1cm)
node[
    currarrow,
    pos=0.5, 
    xscale=-1,
    sloped,
    scale=1] {};
\draw [purple] (-1,0) arc (180:360:1cm);
\end{scope}
\begin{scope}[shift={(6,0)}]
\draw (0,0) circle (1.5cm);
\draw (-0.4,0) arc (180:360:0.4cm and 0.15cm);
\draw (0.2,-0.1) arc (0:180:0.2cm and 0.1cm);
\draw [dashed, teal](0.2,-0.1) arc (180:360:0.65cm and 0.3cm);
\draw [teal](0.2,-0.1) arc (180:0:0.65cm and 0.3cm)
node[
    currarrow,
    pos=0.5, 
    xscale=1,
    sloped,
    scale=1] {};
\draw [purple] (-1,0) arc (180:0:1cm)
node[
    currarrow,
    pos=0.5, 
    xscale=1,
    sloped,
    scale=1] {};
\draw [purple] (-1,0) arc (180:360:1cm);
\end{scope}
\node (a) at (-2,0) {$S_{1}$};
\node (a) at (3,6) {$S_{2}$};
\node (a) at (8,0) {$S_{3}$};
\end{tikzpicture}
\]
\caption{Symmetric embeddings of the copies of $S$ in the central surface, Case 1}
\label{embedsurfcase1}
\end{figure}

Notice that the red and blue curves do not depend on the monodromy. As we can always choose meridians and/or boundary parallel curves as a cut system for $C_{1}$, the red and blue curves only depend on $M$ to that extent. Figure \ref{bluegreencase1} represents these curves when the cut system corresponding to $C_{1}$ consists in a single meridian. The following examples illustrate the construction of the green curves.
\begin{figure}[h!]
\[
\begin{tikzpicture}[scale=0.8]
\node (mypic) at (0,0) {\includegraphics[scale=0.64]{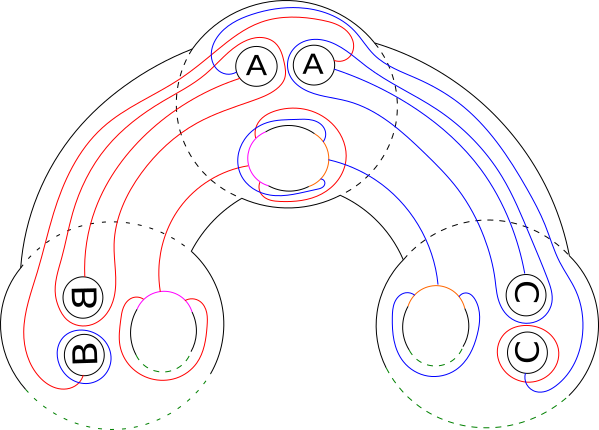}};
\node (a) at (-6.5,-4) {$S_{1}$};
\node (a) at (0,5) {$S_{2}$};
\node (a) at (6.5,-4) {$S_{3}$};
\draw [fill=white] (-1.4,-6) rectangle (1.4,-4);
\node (a) at (0,-5) {Monodromy};
\end{tikzpicture}
\]
\caption{Sets of red and blue curves, Case1\\{\footnotesize Here $S \simeq F_{1,2}$ and the cut system associated to $C_{1}$ is a single meridian; the way the dotted green arcs are identified depends on the monodromy}}
\label{bluegreencase1}
\end{figure}

\begin{Example}
\label{exflip1} 
Consider $M=(S^{2} \times I) \simeq B_{3} \setminus \mathring{B_{1}}$, where $B_{r} \subset \mathbb{R}^3$ is the $3$--ball of radius $r$ centered at the origin. If we divide $M$ along the $xz$ plane, we get a sutured Heegaard splitting of $M$, $M=C_{1} \cup_{S} C_{2}$, of genus $0$, with Heegaard surface an annulus.  
Now define the monodromy $\phi$ as the composition of the rotation of angle $\pi$ along the $z$ axis and the reflection along the sphere of radius $2$ centred at the origin. Then $\phi$ flips the sutured Heegaard splitting $M=C_{1} \cup_{S} C_{2}$ and the boundary components of $M$ and $S$. A diagram featuring the green curves corresponding to this example is drawn on Figure \ref{ex1cas1}.   
\end{Example}
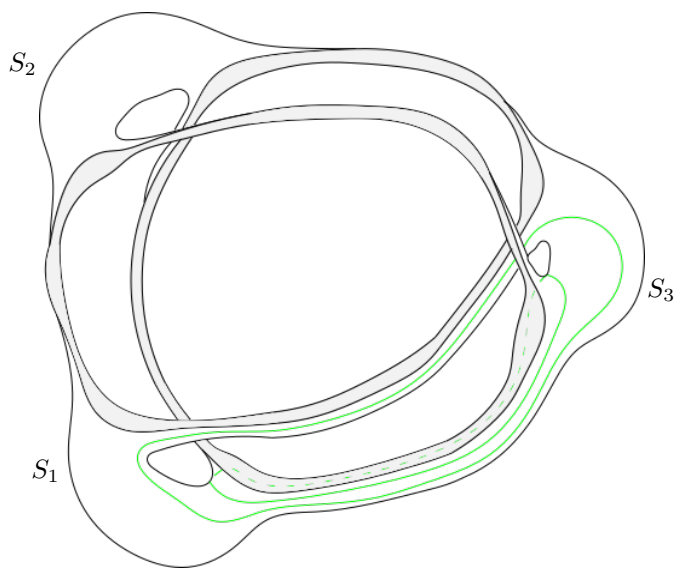
\begin{figure}[h!]
\[
\begin{tikzpicture}[scale=0.6]
\node (mypic) at (0,0) {\includegraphics[scale=0.6]{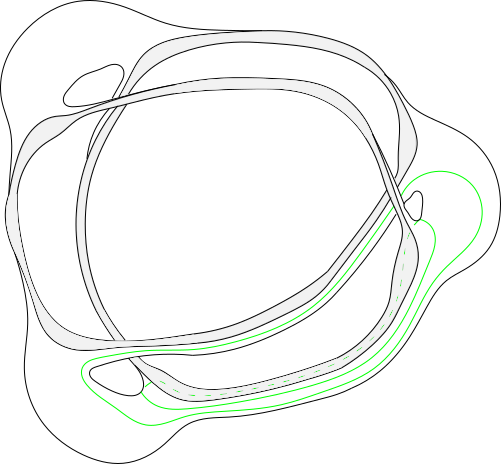}};
\node (a) at (-6.5,-4) {$S_{1}$};
\node (a) at (-7,5) {$S_{2}$};
\node (a) at (7,0) {$S_{3}$};
\end{tikzpicture}
\]
\caption{Set of green curves, Case 1, Example \ref{exflip1}}
\label{ex1cas1}
\end{figure}
\begin{Example}
\label{exflip}
Consider a Heegaard diagram for the genus $1$ Heegaard splitting $H_{1} \cup_{T^{2}} H_{2}$ of the lens space $L(2,1)$. Choose a disk disjoint from the curves on the diagram. Then the sutured Heegaard diagram obtained by removing the interior of the disk corresponds to $L(2,1)$ minus the interior of a $3$--ball $B$. We set this space as our fiber $M$. Let $\phi'$ be a diffeomorphism of $T^{2}$ that sends the meridian $m$ to $-m+2 \ell$. Then $\phi'$ extends to a diffeomorphism of $L(2,1)$, which flips the Heegaard splitting and preserves the $3$--ball $B$. This diffeomorphism restricts to $M$ and flips its sutured Heegaard splitting. We set it as our monodromy $\phi$.

To draw the arcs defined in the algorithm above, we first isotope $m$ and $\ell$ so that they meet the boundary of $S_{1}$ along two segments, then we remove those segments. We obtain two non-isotopic, properly embedded arcs drawn on $S_{1}$. To get their images under the monodromy, we proceed in the same way with the images of the meridian and longitude: $\phi(m) = -m + 2 \ell$ and $\phi(\ell) = \ell$. That gives us two oriented arcs on $S_{3}$. As the image of an arc by the monodromy must be read with opposite direction to comply with the algorithm, we connect each arc on $S_{1}$ to its image on $S_{3}$ with reversed orientation. 

The green curves are displayed on Figure \ref{mono2ex2cas1}. Actually, this works for any lens space minus the interior of a $3$--ball.  
\end{Example} 

\begin{figure}[h!]
\[
\begin{tikzpicture}[scale=0.6]
\node (mypic) at (0,0) {\includegraphics[scale=0.6]{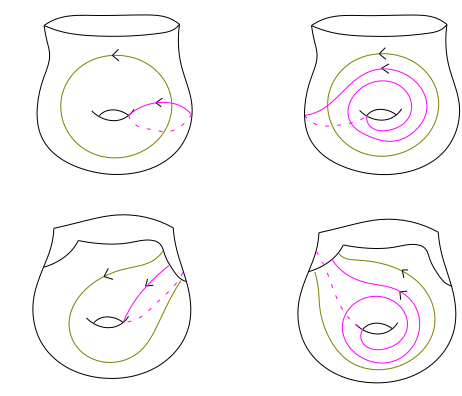}};
\node (a) at (-5.5,2) {$S_{1}$};
\node (a) at (6.5,2) {$S_{3}$};
\node (a) at (-5.5,-2) {$S_{1}$};
\node (a) at (6.5,-2) {$S_{3}$};
\end{tikzpicture}
\]
\caption{Finding the arcs on Example \ref{exflip} \\{\footnotesize Above: two essential oriented curves on $S_{1}$ and their oriented images on $S_{3}$; below: the oriented arcs obtained from these curves}}
\label{mono1ex2cas1}
\end{figure}

\begin{figure}[h!]
\[
\begin{tikzpicture}[scale=0.7]
\node (mypic) at (0,0) {\includegraphics[scale=0.56]{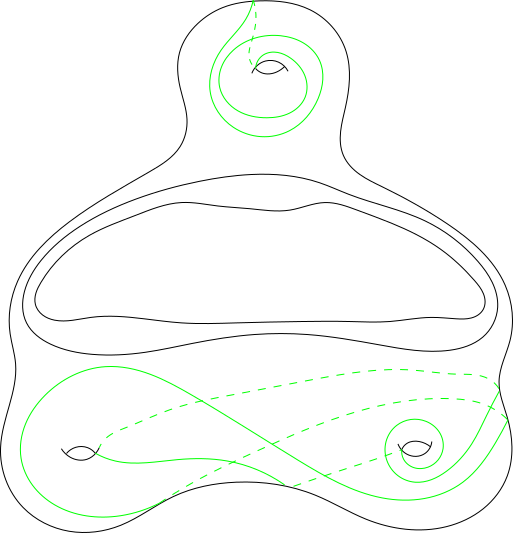}};
\node (a) at (0,6) {$S_{2}$};
\node (a) at (4,-6) {$S_{3}$};
\node (a) at (-4,-6) {$S_{1}$};
\end{tikzpicture}
\]
\caption{Green curves corresponding to Exemple \ref{exflip} \\{\footnotesize The curves on $S_{1}$ and $S_{3}$ are obtained from Figure \ref{mono1ex2cas1} by connecting each arc on $S_{1}$ to its image with reversed orientation on $S_{3}$; the curve on $S_{2}$ is the symmetric of a curve $-m + 2 \ell$ drawn on $S_{1}$ }}
\label{mono2ex2cas1}
\end{figure}

\subsection{Relative trisection diagrams, Case 2}
\label{rtdcase2}
As we proceed essentially in the same way as in the previous case, we just outline the differences: we divide the plane in six sectors, using six rays $\theta_{k}$ from the origin; we draw one copy of $S$ on each ray; we start with $S_{1}$, then draw $S_{2}$ as its symmetric along a ray midway between $\theta_{1}$ and $\theta_{2}$, and so on. With this choice of embedding, we can use the cut system for the product compression bodies $C_{g,b}$ defined in \ref{productcompbody}. Note that $S_{1}$ and $S_{6}$, between which the monodromy occurs, are now symmetric, in contrast to Case 1.

As the monodromy preserves the Heegaard splitting, we can get examples just by setting it to be the identity. Figure \ref{curvescase2} represents a relative trisection diagram for $X=M \times S^{1}$, when the fiber $M$ admits the sutured Heegaard diagram represented on Figure \ref{basediagramcase2}. If we take instead $M$ to be a knot exterior in $S^{3}$, we recover the relative trisection diagram for the product of a knot exterior with $S^{1}$ in \cite{castro2018diagrams}. 
\begin{figure}[h!]
\[
\begin{tikzpicture}
\node (mypic) at (0,0) {\includegraphics[scale=0.56]{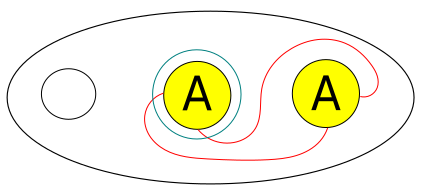}};
\end{tikzpicture}
\]
\caption{A sutured Heegaard diagram of the fiber}
\label{basediagramcase2}
\end{figure}

\begin{figure}[h!]
\[
\begin{tikzpicture}[scale=0.8]
\draw (0,0) circle (2cm);
\draw (0,0) circle (8cm);
\draw (-5,0) circle (0.5cm);
\draw (-7,0) circle (0.5cm);
\node (a) at (-5,0) {\reflectbox{A}};
\node (a) at (-7,0) {A};
\draw [purple] (-2.75,-0.43) arc (-60:60:0.5cm);
\draw [orange] (-3.25,0.43) arc (120:240:0.5cm);
\draw [black] (-2.75,0.43) arc (60:120:0.5cm);
\draw [black] (-3.25,-0.43) arc (240:300:0.5cm);
\begin{scope}[rotate=60]
\draw (-5,0) circle (0.5cm);
\draw (-7,0) circle (0.5cm);
\node[rotate=60] (a) at (-5,0) {\reflectbox{B}};
\node[rotate=60] (a) at (-7,0) {B};
\draw [pink] (-2.75,-0.43) arc (-60:60:0.5cm);
\draw [orange] (-3.25,0.43) arc (120:240:0.5cm);
\draw [black] (-2.75,0.43) arc (60:120:0.5cm);
\draw [black] (-3.25,-0.43) arc (240:300:0.5cm);
\end{scope}
\begin{scope}[rotate=120]
\draw (-5,0) circle (0.5cm);
\draw (-7,0) circle (0.5cm);
\node[rotate=120](a) at (-5,0) {\reflectbox{C}};
\node[rotate=120] (a) at (-7,0) {C};
\draw [pink] (-2.75,-0.43) arc (-60:60:0.5cm);
\draw [teal] (-3.25,0.43) arc (120:240:0.5cm);
\draw [black] (-2.75,0.43) arc (60:120:0.5cm);
\draw [black] (-3.25,-0.43) arc (240:300:0.5cm);
\end{scope}
\begin{scope}[rotate=180]
\draw (-5,0) circle (0.5cm);
\draw (-7,0) circle (0.5cm);
\node[rotate=180] (a) at (-5,0) {\reflectbox{D}};
\node[rotate=180] (a) at (-7,0) {D};
\draw [brown] (-2.75,-0.43) arc (-60:60:0.5cm);
\draw [teal] (-3.25,0.43) arc (120:240:0.5cm);
\draw [black] (-2.75,0.43) arc (60:120:0.5cm);
\draw [black] (-3.25,-0.43) arc (240:300:0.5cm);
\end{scope}
\begin{scope}[rotate=240]
\draw (-5,0) circle (0.5cm);
\draw (-7,0) circle (0.5cm);
\node[rotate=-120] (a) at (-5,0) {\reflectbox{E}};
\node[rotate=-120] (a) at (-7,0) {E};
\draw [brown] (-2.75,-0.43) arc (-60:60:0.5cm);
\draw [yellow] (-3.25,0.43) arc (120:240:0.5cm);
\draw [black] (-2.75,0.43) arc (60:120:0.5cm);
\draw [black] (-3.25,-0.43) arc (240:300:0.5cm);
\end{scope}
\begin{scope}[rotate=300]
\draw (-5,0) circle (0.5cm);
\draw (-7,0) circle (0.5cm);
\node[rotate=-60] (a) at (-5,0) {\reflectbox{F}};
\node[rotate=-60] (a) at (-7,0) {F};
\draw [purple] (-2.75,-0.43) arc (-60:60:0.5cm);
\draw [yellow] (-3.25,0.43) arc (120:240:0.5cm);
\draw [black] (-2.75,0.43) arc (60:120:0.5cm);
\draw [black] (-3.25,-0.43) arc (240:300:0.5cm);
\end{scope}
\draw [green] (4:7) arc [radius=7, start angle=4, end angle=56];
\draw [green] (5.5:5) arc [radius=5, start angle=5.5, end angle=54.5];
\draw [green] (0:4.3) arc [radius=4.3, start angle=0, end angle=60];
\draw [green] (0:5.7) arc [radius=5.7, start angle=0, end angle=60];
\draw [green] (4.3,0) arc (180:360:0.7cm);
\draw [green] (2.15,3.72) arc (240:60:0.7cm);
\draw [green] (2:2.3) arc [radius=2.3, start angle=2, end angle=58];
\draw [green] (2.3,0.1) arc (180:270:0.2cm);
\draw [green] (1.24,1.95) arc (270:120:0.14cm);
\begin{scope}[rotate=60] 
\draw [red] (4:7) arc [radius=7, start angle=4, end angle=56];
\draw [red] (5.5:5) arc [radius=5, start angle=5.5, end angle=54.5];
\draw [red] (0:4.3) arc [radius=4.3, start angle=0, end angle=60];
\draw [red] (0:5.7) arc [radius=5.7, start angle=0, end angle=60];
\draw [red] (4.3,0) arc (180:360:0.7cm);
\draw [red] (2.15,3.72) arc (240:60:0.7cm);
\draw [red] (2:3.7) arc [radius=3.7, start angle=2, end angle=58];
\draw [red] (3.7,0.15) arc (0:-90:0.2cm);
\draw [red] (2,3.12) arc (50:180:0.14cm);
\end{scope}
\begin{scope}[rotate=120] 
\draw [blue] (4:7) arc [radius=7, start angle=4, end angle=56];
\draw [blue] (5.5:5) arc [radius=5, start angle=5.5, end angle=54.5];
\draw [blue] (0:4.3) arc [radius=4.3, start angle=0, end angle=60];
\draw [blue] (0:5.7) arc [radius=5.7, start angle=0, end angle=60];
\draw [blue] (4.3,0) arc (180:360:0.7cm);
\draw [blue] (2.15,3.72) arc (240:60:0.7cm);
\draw [blue] (2.55,-0.2) .. controls (2,-0.2) and (3.5,-1.5) .. (3.6,0);
\draw [blue] (2.55,0.2) .. controls (2,0.2) and (3.5,1.5) .. (3.6,0);
\draw [blue] (2:2.3) arc [radius=2.3, start angle=2, end angle=58];
\draw [blue] (2.3,0.1) arc (180:270:0.2cm);
\draw [blue] (1.24,1.95) arc (270:120:0.14cm);
\end{scope}
\begin{scope}[rotate=180] 
\draw [green] (4:7) arc [radius=7, start angle=4, end angle=56];
\draw [green] (5.5:5) arc [radius=5, start angle=5.5, end angle=54.5];
\draw [green] (0:4.3) arc [radius=4.3, start angle=0, end angle=60];
\draw [green] (0:5.7) arc [radius=5.7, start angle=0, end angle=60];
\draw [green] (4.3,0) arc (180:360:0.7cm);
\draw [green] (2.15,3.72) arc (240:60:0.7cm);
\draw [green] (2:3.7) arc [radius=3.7, start angle=2, end angle=58];
\draw [green] (3.7,0.15) arc (0:-90:0.2cm);
\draw [green] (2,3.12) arc (50:180:0.14cm);
\end{scope}
\begin{scope}[rotate=240]
\draw [red] (4:7) arc [radius=7, start angle=4, end angle=56];
\draw [red] (5.5:5) arc [radius=5, start angle=5.5, end angle=54.5];
\draw [red] (0:4.3) arc [radius=4.3, start angle=0, end angle=60];
\draw [red] (0:5.7) arc [radius=5.7, start angle=0, end angle=60];
\draw [red] (4.3,0) arc (180:360:0.7cm);
\draw [red] (2.15,3.72) arc (240:60:0.7cm);
\draw [red] (2:2.3) arc [radius=2.3, start angle=2, end angle=58];
\draw [red] (2.3,0.1) arc (180:270:0.2cm);
\draw [red] (1.24,1.95) arc (270:120:0.14cm);
\end{scope}
\begin{scope}[rotate=300]
\draw [blue] (4:7) arc [radius=7, start angle=4, end angle=56];
\draw [blue] (5.5:5) arc [radius=5, start angle=5.5, end angle=54.5];
\draw [blue] (0:4.3) arc [radius=4.3, start angle=0, end angle=60];
\draw [blue] (0:5.7) arc [radius=5.7, start angle=0, end angle=60];
\draw [blue] (4.3,0) arc (180:360:0.7cm);
\draw [blue] (2.15,3.72) arc (240:60:0.7cm);
\draw [blue] (2:3.7) arc [radius=3.7, start angle=2, end angle=58];
\draw [blue] (3.7,0.15) arc (0:-90:0.2cm);
\draw [blue] (2,3.12) arc (50:180:0.14cm);
\end{scope}
\draw [green] (2.55,-0.2) .. controls (2,-0.2) and (3.5,-1.5) .. (3.6,0);
\draw [green] (2.55,0.2) .. controls (2,0.2) and (3.5,1.5) .. (3.6,0);
\begin{scope}[rotate=60]
\draw [green] (2.55,-0.2) .. controls (2,-0.2) and (3.5,-1.5) .. (3.6,0);
\draw [green] (2.55,0.2) .. controls (2,0.2) and (3.5,1.5) .. (3.6,0);
\end{scope}
\begin{scope}[rotate=240]
\draw [green] (3.45,-0.2) .. controls (3.9,-0.2) and (2.47,-1.1) .. (2.45,0);
\draw [green] (3.45,0.2) .. controls (3.9,0.2) and (2.47,1.1) .. (2.45,0);
\end{scope}
\begin{scope}[rotate=180]
\draw [green] (3.45,-0.2) .. controls (3.9,-0.2) and (2.47,-1.1) .. (2.45,0);
\draw [green] (3.45,0.2) .. controls (3.9,0.2) and (2.47,1.1) .. (2.45,0);
\end{scope}
\begin{scope}[rotate=180]
\draw [blue] (2.55,-0.2) .. controls (2,-0.2) and (3.5,-1.5) .. (3.6,0);
\draw [blue] (2.55,0.2) .. controls (2,0.2) and (3.5,1.5) .. (3.6,0);
\end{scope}
\begin{scope}[rotate=60]
\draw [red] (3.45,-0.2) .. controls (3.9,-0.2) and (2.47,-1.1) .. (2.45,0);
\draw [red] (3.45,0.2) .. controls (3.9,0.2) and (2.47,1.1) .. (2.45,0);
\end{scope}
\begin{scope}[rotate=120]
\draw [red] (3.45,-0.2) .. controls (3.9,-0.2) and (2.47,-1.1) .. (2.45,0);
\draw [red] (3.45,0.2) .. controls (3.9,0.2) and (2.47,1.1) .. (2.45,0);
\end{scope}
\begin{scope}[rotate=240]
\draw [red] (2.55,-0.2) .. controls (2,-0.2) and (3.5,-1.5) .. (3.6,0);
\draw [red] (2.55,0.2) .. controls (2,0.2) and (3.5,1.5) .. (3.6,0);
\begin{scope}[rotate=60]
\draw [red] (2.55,-0.2) .. controls (2,-0.2) and (3.5,-1.5) .. (3.6,0);
\draw [red] (2.55,0.2) .. controls (2,0.2) and (3.5,1.5) .. (3.6,0);
\end{scope}
\end{scope}
\begin{scope}[rotate=300]
\draw [blue] (3.45,-0.2) .. controls (3.9,-0.2) and (2.47,-1.1) .. (2.45,0);
\draw [blue] (3.45,0.2) .. controls (3.9,0.2) and (2.47,1.1) .. (2.45,0);
\end{scope}
\draw [blue] (3.45,-0.2) .. controls (3.9,-0.2) and (2.47,-1.1) .. (2.45,0);
\draw [blue] (3.45,0.2) .. controls (3.9,0.2) and (2.47,1.1) .. (2.45,0);
\node (a) at (240:8.5) {$S_{1}$};
\node (a) at (180:8.5) {$S_{2}$};
\node (a) at (120:8.5) {$S_{3}$};
\node (a) at (60:8.5) {$S_{4}$};
\node (a) at (0:8.5) {$S_{5}$};
\node (a) at (300:8.5) {$S_{6}$};
\draw [blue] (60:7) circle (0.7cm);
\draw [red] (180:7) circle (0.7cm);
\draw [green] (300:7) circle (0.7cm);
\draw [blue] (245.5:5.1) .. controls (260:6) and (250:8) .. (240:7.5);
\draw [blue] (244:6.9) .. controls (250:6.8) and (250:6.25) .. (240:6) .. controls (232:5.75) and (230.5:5.25) .. (230:5) .. controls (231:4.6) and (236:4.45) .. (240:4.5);
\begin{scope}[rotate=120]
\draw [red] (245.5:5.1) .. controls (260:6) and (250:8) .. (240:7.5);
\draw [red] (244:6.9) .. controls (250:6.8) and (250:6.25) .. (240:6) .. controls (232:5.75) and (230.5:5.25) .. (230:5) .. controls (231:4.6) and (236:4.45) .. (240:4.5);
\end{scope}
\begin{scope}[rotate=240]
\draw [green] (245.5:5.1) .. controls (260:6) and (250:8) .. (240:7.5);
\draw [green] (244:6.9) .. controls (250:6.8) and (250:6.25) .. (240:6) .. controls (232:5.75) and (230.5:5.25) .. (230:5) .. controls (231:4.6) and (236:4.45) .. (240:4.5);
\end{scope}
\end{tikzpicture}
\]
\caption{A relative trisection diagram of $X=S^{1} \times M$, where $M$ is defined by the sutured Heegaard diagram of Figure \ref{basediagramcase2}}
\label{curvescase2}
\end{figure}
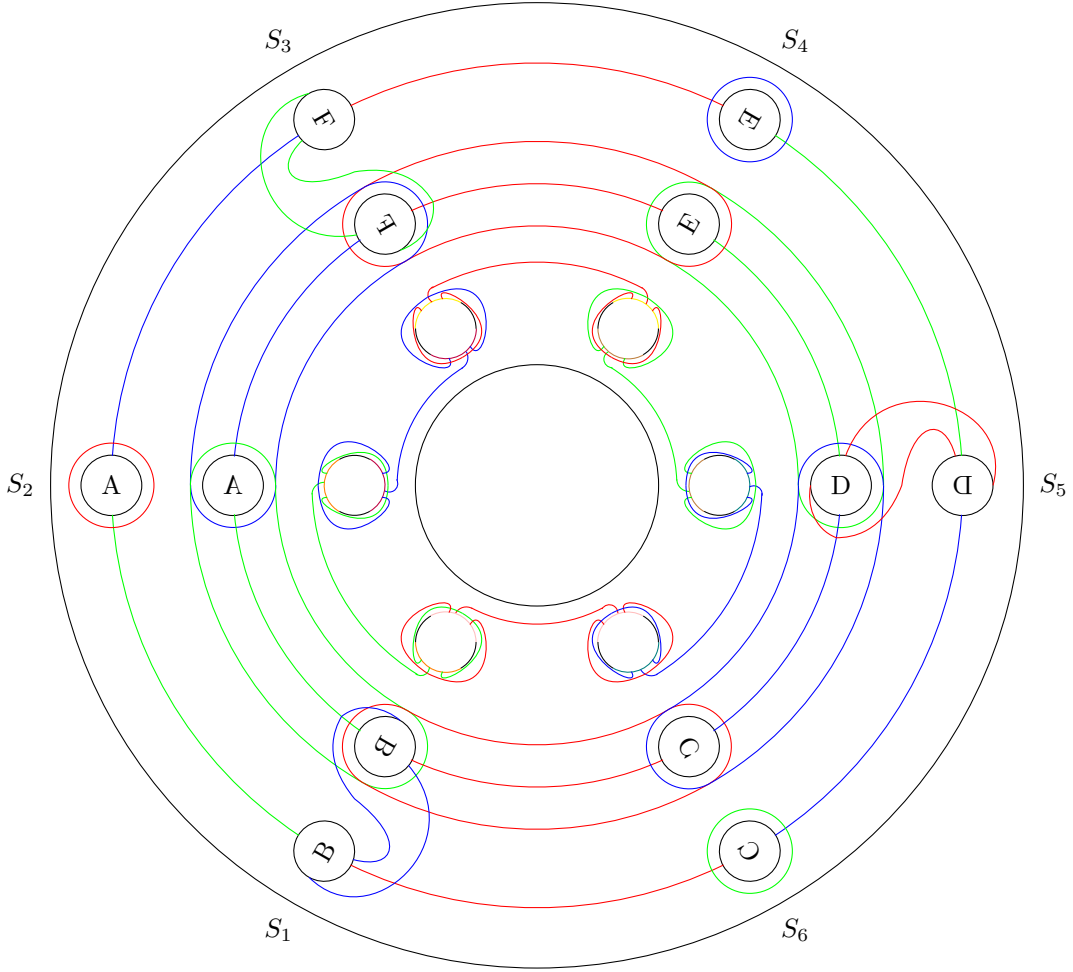
\begin{Remark}
In \cite{koenig2017trisections}, the trisection diagram obtained when the monodromy preserved the Heegaard splitting of the fiber could be destabilised. It would be interesting to see if the relative trisection diagram (of higher genus) obtained in the compact case can be destabilised as well.
\end{Remark} 
\section{Applications}
Relative trisections can be combined to produce trisected closed manifolds. This aspect of the theory was developed in \cite{castro2018trisections} and \cite{castro2017trisections}. In what follows, we consider only manifolds with connected boundaries. Our goal, in this section, is to use the results established in \cite{castro2017trisections} in order to glue  relatively trisected bundles over the circle to other relatively trisected $4$--manifolds, thus producing classes of examples of trisected (closed) $4$--manifolds. We recover the trisected closed bundles over the circle of \cite{koenig2017trisections}, the trisected spun manifolds of \cite{meier2017trisections}, and we also provide trisections for some classes of $4$--dimensional open-books. 
\subsection{Gluing relative trisection diagrams: theoretical aspects}  
\begin{Definition}
\label{defcutsysarcs} 
Let $S$ be a compact surface and $\alpha$ a cut system of curves on $S$. Denote by $S_{\alpha}$ the surface obtained by surgering $S$ along the $\alpha$ curves (that is, by cutting $S$ along the $\alpha$ curves and gluing back a disk along every newly created boundary component). An \emph{arc system} associated to $\alpha$ is a family of properly embedded arcs in $S$, such that cutting $S_{\alpha}$ along $A_{\alpha}$ produces a disk. We say that two cut systems together with their associated arcs $(\alpha, A_{\alpha})$ and $(\alpha ', A_{\alpha '})$ are \emph{handleslide equivalent} if one can be obtained from the other by sliding curves over curves and arcs over curves.  
\end{Definition}
\begin{Definition}
\label{defarccutsys}
An arced diagram is a tuple $(S; \alpha, \beta, \gamma, A_{\alpha} , A_{\beta}, A_{\gamma})$ such that:
\begin{itemize}
\item $(S; \alpha, \beta, \gamma)$ is a relative trisection diagram;
\item $A_{\alpha}$ (resp. $A_{\beta}$, $A_{\gamma}$) is an arc system associated to $\alpha$ (resp. $\beta$, $\gamma$);
\item $(S; \alpha, \beta , A_{\alpha} , A_{\beta})$ is handleslide equivalent to some $(S; \alpha ', \beta ' , A_{\alpha '} , A_{\beta '})$, where $(S; \alpha ', \beta ')$ is diffeomorphic to a standard sutured Heegaard diagram $D$, and $A_{\alpha '} = A_{\beta '}$;
\item $(S;  \beta , \gamma, A_{\beta} , A_{\gamma})$ is handleslide equivalent to some $(S; \beta ', \gamma ' , A_{\beta '} , A_{\gamma '})$, where $(S; \beta ', \gamma ')$ is diffeomorphic to the same standard sutured Heegaard diagram $D$ and $A_{\beta '} = A_{\gamma '}$.
\end{itemize} 
\end{Definition}
\begin{Remark}
\label{rmkarc}
This intrinsic definition follows \cite{gay2022doubly}. To explicitly draw an arced diagram from a relative trisection diagram, we apply the algorithm described in \cite{castro2018diagrams}.
\begin{itemize}
\item Choose an arc system $A_\alpha$ associated to $\alpha$.
\item There is a collection of arcs $A_{\beta}$, obtained from $A_\alpha$ by sliding $A_\alpha$ arcs over $\alpha$ curves, and a set of curves $\tilde{\beta}$, obtained from $\beta$ by handleslides, such that $A_\beta$ and $\tilde{\beta}$ are disjoint. 
\item There is a collection of arcs $A_{\gamma}$, obtained from $A_\beta$ by sliding $A_\beta$ arcs over $\tilde{\beta}$ curves, and a set of curves $\tilde{\gamma}$, obtained from $\gamma$ by handleslides, such that $A_\gamma$ and $\tilde{\gamma}$ are disjoint.
\end{itemize}
Then $(S; \alpha, \tilde{\beta} , \tilde{\gamma}, A_{\alpha}, A_{\beta}, A_{\gamma})$ is an arced diagram obtained from $(S; \alpha, \beta, \gamma)$.

That we can always find the advertised arcs and curves follows directly from the fact that $(S, \alpha, \beta)$, $(S, \beta, \gamma)$ and $(S, \alpha, \gamma)$ are standard. We make many choices when constructing the arcs and curves systems, but (\cite{castro2018trisections}) these choices do not affect the final result: two arced diagrams obtained from the same relative trisection diagram are handle\-slide equivalent. 
\end{Remark}
\begin{Remark}
The diagram $(S; \alpha, \beta, \gamma)$ describes a relatively trisected $4$--manifold. The relative trisection induces an open-book decomposition of the boundary of the manifold, and the monodromy of this open-book can be explicitly computed using an associated arced diagram (see \cite{castro2018diagrams} or \cite{castro2018trisections}).
\end{Remark}
Let $X$ and $X'$ be relatively trisected compact $4$--manifolds, with associated re\-lative trisection diagrams $(S; \alpha, \beta, \gamma)$ and $(S; \alpha ', \beta ', \gamma ')$, inducing the open book decomposition $\mathcal{OB}$ on $\partial X$ and $\mathcal{OB'}$ on $\partial X'$. 
Suppose that there exists an orientation reversing diffeomorphism $f$ from $\partial X$ to $\partial X'$ that takes $\mathcal{OB}$ to $\mathcal{OB'}$.\\
Construct an arced diagram $(S; \alpha,\tilde{\beta}, \tilde{\gamma} , A_\alpha , A_\beta , A_\gamma )$ from $(S; \alpha, \beta, \gamma)$. As $f(A_\alpha )$ cuts $S_{\alpha'} = f(S_\alpha )$ into a disk, construct an arced diagram $(S; \alpha ',\tilde{\beta '}, \tilde{\gamma '} , A_{\alpha '}, A_{\beta '} , A_{\gamma  '})$ from $(S, \alpha ', \beta ', \gamma ')$, starting from $A_{\alpha '}=f(A_\alpha ) $ and following the process described in \ref{rmkarc}.\\
Glue $S$ and $S'$ along their boundary, according to $f$. The process above associates to every arc in $A_\alpha$ (resp. $A_\beta$, $A_\gamma$) an arc in $A_{\alpha '}$ (resp. $A_{\beta '}$, $A_{\gamma '}$). Connect every arc to its associate, thus creating three families of curves $A_\alpha \cup A_{\alpha '}$, $A_\beta \cup A_{\beta '}$, $A_\gamma \cup A_{\gamma '}$ on $S \cup_{f} S'$.
\begin{Theorem}[Castro--Gay--Pinz\'{o}n-Caicedo]
\label{thmgluerel}
With the notations of the discussion above, $\big(S \cup_{f} S' ; \alpha \cup \alpha' \cup (A_\alpha \cup A_{\alpha '}), \tilde{\beta} \cup \tilde{\beta '} \cup (A_\beta \cup A_{\beta '}), \tilde{\gamma} \cup \tilde{\gamma '} \cup (A_\gamma \cup A_{\gamma '}) \big)$ is a trisection diagram for $X \cup_{f} X'$.
\end{Theorem}
\begin{Remark}
\label{rmkglu}
Proofs of Theorem \ref{thmgluerel} can be found in \cite{castro2017trisections} and \cite{castro2018diagrams}. More precisely, it is shown that gluing $(g,k,p,b)$ and $(g',k',p',b')$ relative trisections produces a $\big( g+g'+b-1,k+k'- (2p+b-1) \big)$--trisection of $X \cup_{f} X'$, in the case of connected boundaries.
\end{Remark}
\subsection{Examples}
We now use Theorem \ref{thmgluerel} to construct three classes of trisection diagrams from re\-lative trisection diagrams. By doing so, we recover some trisections diagram of \cite{koenig2017trisections} and \cite{meier2017trisections}, and we also produce some classes of trisection diagrams of $4$--dimensional open-books, that, to our current knowledge, do not yet exist in the literature.
\subsubsection{Koenig's trisections of bundles over the circle}
We consider a closed bundle over the circle of fiber $M$, where the monodromy preserves a genus $g$ Heegaard splitting of the fiber. Consider the compact $3$--manifold $M^o = M \setminus Int(B^{3})$, where $B^3$ is a $3$--ball transverse to the Heegaard surface of the splitting and preserved by the monodromy. Then the monodromy descends to $M^o$ and we obtain a bundle over $S^1$ of fiber $M^o$, of which we can draw a relative trisection diagram as in Section 5. Actually, because the boundary of $M^o$ is just $S^{2}$, we can use a lower genus (unbalanced) relative trisection of the bundle based on the schematic of Figure \ref{decomps1Mo}. An example of a relative trisection diagram corresponding to this decomposition is featured on Figure \ref{arceddiags1Mo}, with the associated systems of arcs (on the diagram, we assume the monodromy to be the identity, but this doesn't affect our future conside\-rations. We start from the Heegaard diagram for $L(2,1)$ featured on Figure \ref{hdl21}, left; by removing an open disk disjoint from the curves on the Heegaard surface, we obtain a sutured Heegaard diagram for $L(2,1)^o$). To obtain our closed bundle over the circle of fiber $M$, we must glue back in a relatively trisected $S^1 \times B^{3}$, of which an arced diagram is just a cylinder with three parallel arcs as in Figure \ref{arceddiags1b3}. As the pages and arcs defined by these diagrams agree, all we need to do is to glue the two arced diagrams together according to the identity. We obtain a $(4g+1)$--trisection diagram of the closed bundle (see Figure \ref{koenigdiagramg5}), equivalent to the diagram featured in \cite{koenig2017trisections}. This diagram can be destabilised $g$ times to obtain $(3g+1)$--trisection diagram (for that last point, see \cite{koenig2017trisections}).   
\begin{figure}[h!]
\[
\begin{tikzpicture}[scale=0.5]
\node (mypic) at (0,0) {\includegraphics[scale=0.5]{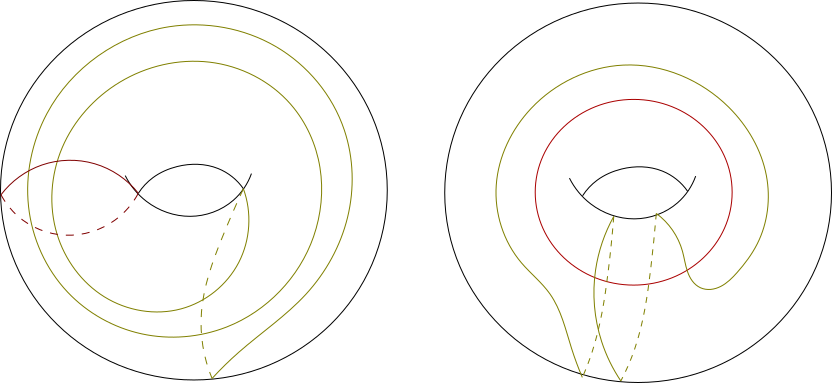}};
\end{tikzpicture}
\]
\caption{Two Heegaard diagrams for $L(2,1)$}
\label{hdl21}
\end{figure}

\begin{figure}[h!]
\[
\begin{tikzpicture}[scale=0.6]
\draw (0,0) -- (8,0) -- (8,4) -- (0,4) -- (0,0) -- cycle;
\draw (0,1.9) -- (8,1.9);
\draw (0,2.1) -- (8,2.1);
\draw (2,2.1) -- (2,4);
\draw (3,0) -- (3,1.9);
\draw (5,0) -- (5,1.9);
\draw (6,2.1) -- (6,4);
\node (a) at (0.6,0.6) {$X_{3}$};
\node (a) at (3.6,0.6) {$X_{1}$};
\node (a) at (7.4,0.6) {$X_{3}$};
\node (a) at (0.6,3.6) {$X_{1}$};
\node (a) at (3.6,3.6) {$X_{2}$};
\node (a) at (7.4,3.6) {$X_{1}$};
\end{tikzpicture}
\]
\caption{A trisection of the bundle over $S^1$ of fiber $M^o$}
\label{decomps1Mo}
\end{figure}

\begin{figure}[h!]
\[
\begin{tikzpicture}[scale=0.5]
\node (mypic) at (0,0) {\includegraphics[scale=0.5]{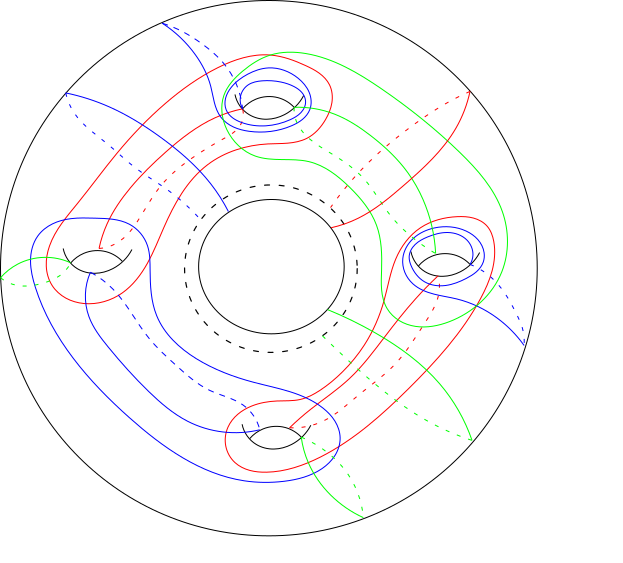}};
\end{tikzpicture}
\]
\caption{A genus $4$ arced diagram for $S^1 \times M^o$, when $M \simeq L(2,1)$}
\label{arceddiags1Mo}
\end{figure}

\begin{figure}[h!]
\[
\begin{tikzpicture}[scale=0.4]
\node (mypic) at (0,0) {\includegraphics[scale=0.6]{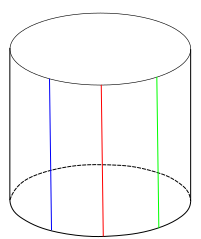}};
\end{tikzpicture}
\]
\caption{An arced diagram for $S^{1} \times B^3$}
\label{arceddiags1b3}
\end{figure}

\begin{figure}[h!]
\[
\begin{tikzpicture}[scale=0.6]
\node (mypic) at (0,0) {\includegraphics[scale=0.5]{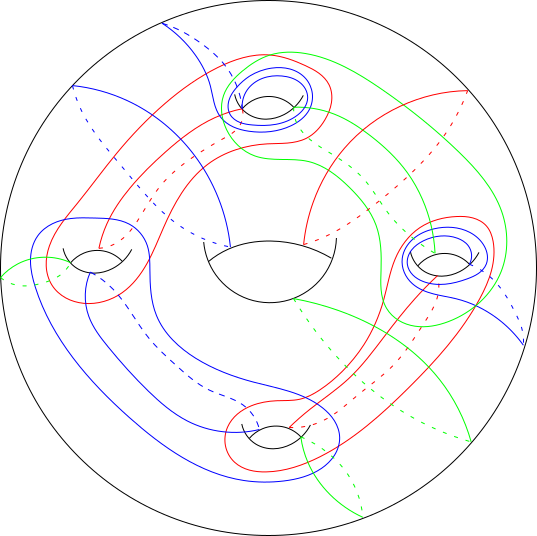}};
\end{tikzpicture}
\]
\caption{A trisection diagram for $S^1 \times M$, when $M \simeq L(2,1)$}
\label{koenigdiagramg5}
\end{figure}

\subsubsection{Meier's trisections of spun manifolds}
Let $M$ be a closed $3$--manifold and $M^o = M \setminus Int(B^{3})$. The spun of $M$ is defined as $\mathcal{S}(M) = (M^o \times S^{1}) \cup_{S^{2} \times S^{1}, Id} (S^{2} \times D^{2})$. We proceed as in the previous example to draw a sutured Heegaard diagram for $M^o$ from a genus $g$ Heegaard diagram for $M$. Then we apply the procedure defined in \ref{rtdcase2} to obtain a relative trisection diagram for $M^o \times S^{1}$, and finally an arced diagram for $M^o \times S^{1}$ (an example is featured on Figure \ref{arceddiagg6d2}, using the Heegaard diagram for $M \simeq L(2,1)$ on Figure \ref{hdl21}, right). An arced diagram for $S^{2} \times D^{2}$ is featured on Figure \ref{arceddiags2d2} (see \cite{castro2018diagrams}). As the pages and arcs of these diagrams agree, we can glue them to obtain a $(6g+3)$--trisection diagram for $\mathcal{S}(M)$. Interestingly, we can explicitly modify this diagram through handleslides and three destabilizations to obtain a diagram that is just the initial relative diagram of $M^o \times S^{1}$ with two disks capping off its boundary components (see Figure \ref{trisecdiagspung6} for an example). Then we can further modify this trisection diagram through handleslides and $3g$ destabilisations to obtain the diagram featured in \cite{meier2017trisections}.

\begin{figure}[h!]
\[
\begin{tikzpicture}[scale=0.4]
\node (mypic) at (0,0) {\includegraphics[scale=0.4]{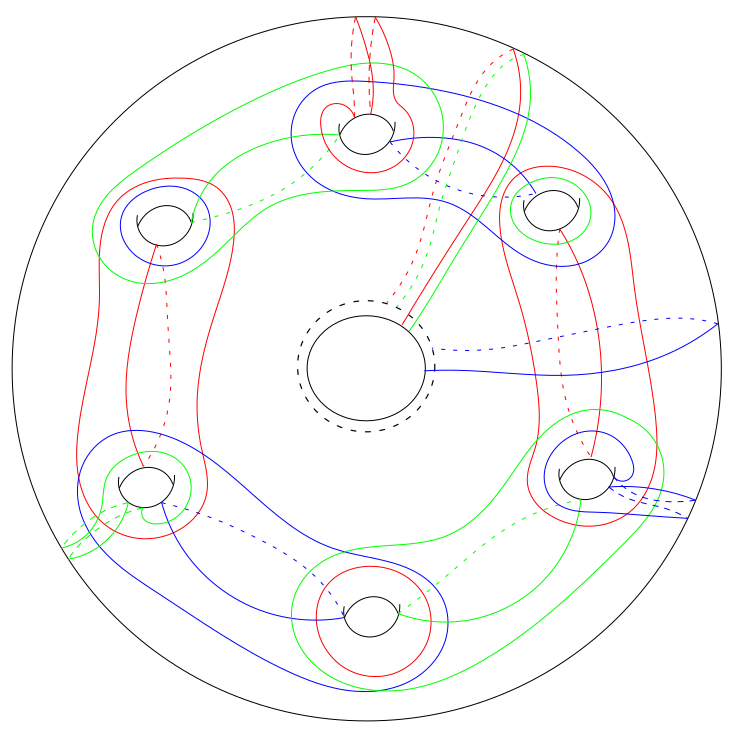}};

\end{tikzpicture}
\]
\caption{A genus $6$ arced diagram for $S^1 \times M^o$, when $M \simeq L(2,1)$}
\label{arceddiagg6d2}
\end{figure}

\begin{figure}[h!]
\[
\begin{tikzpicture}[scale=0.3]
\node (mypic) at (0,0) {\includegraphics[scale=0.4]{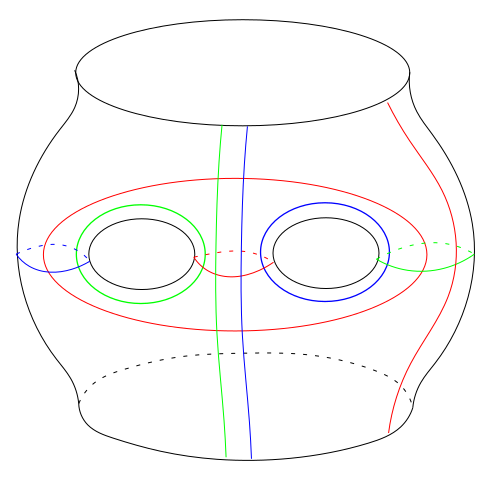}};

\end{tikzpicture}
\]
\caption{An arced diagram for $S^2 \times D^2$\\{\footnotesize We made some handleslides of arcs to prepare for the future destabilizations of the glued diagram.}}
\label{arceddiags2d2}
\end{figure}

\begin{figure}[h!]
\[
\begin{tikzpicture}[scale=0.4]
\node (mypic) at (0,0) {\includegraphics[scale=0.4]{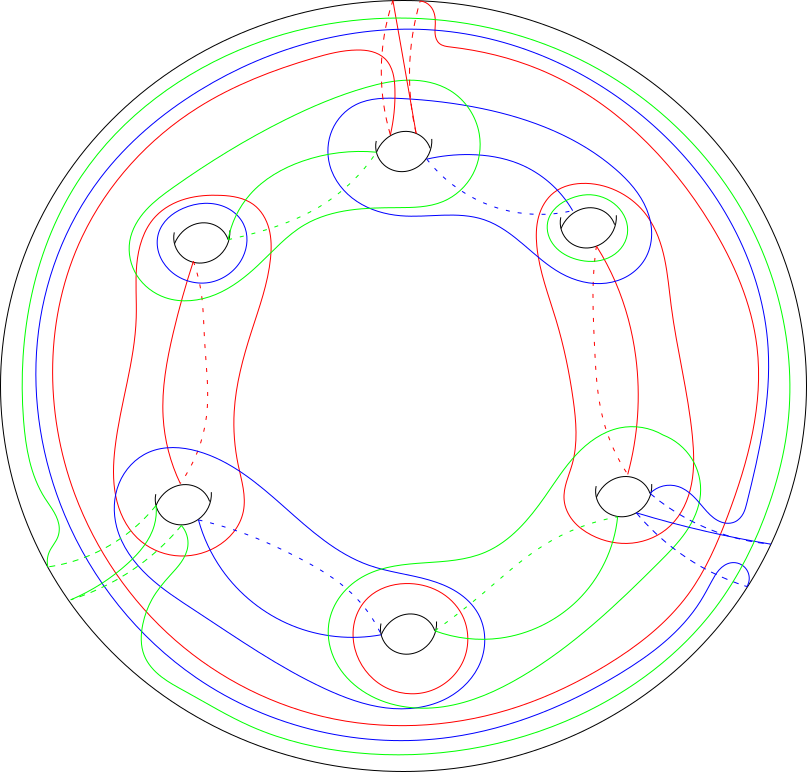}};

\end{tikzpicture}
\]
\caption{A trisection diagram for $\mathcal{S}(M)$, when $M \simeq L(2,1)$}
\label{trisecdiagspung6}
\end{figure}

\subsubsection{Trisections of some $4$--dimensional open-books}
\label{trisectionopenbooks}
Let $M$ be a compact $3$--manifold with boundary a torus, and $\phi$ a self-diffeomorphism of $M$ that is the identity on a regular neighbourhood of $\partial M$. We define the $4$--dimensional open-book of fiber $M$ and monodromy $\phi$ as:
\begin{equation*}
\mathcal{OB}(M) = \big( (M \times I ) / \sim \big) \cup_{\partial M \times S^{1}, Id} \big( \partial M \times D^2 \big), \text{where} \ (x,0) \sim (\phi(x) , 1).
\end{equation*} 
We prove the following. 
\begin{Theorem}
\label{openbookth}
Let $\mathcal{OB}(M)$ be a $4$--dimensional open-book of fiber $M$ and mono\-dromy $\phi$, where $M$ is a compact manifold with boundary a torus. Then there exists a sutured Heegaard splitting of $M$, preserved by $\phi$, that induces a decomposition of $\partial M$ as $\partial M = ( \partial M \setminus \mathring{D}) \cup D$, whith $D$ a disk in $\partial M$. Further, $\mathcal{OB}(M)$ admits a $(6g+4)$--trisection, where $g$ is the genus of this sutured Heegaard splitting of $M$. Moreover, a trisection diagram for $\mathcal{OB}(M)$ can be explicitly derived from a sutured Heegaard diagram corresponding to this splitting.    
\end{Theorem}
\begin{proof}
To produce a trisection of $\mathcal{OB}(M)$, we glue an arced diagram for the bundle $\big( (M \times I ) / \sim \big)$ to an arced diagram of $T^{2} \times D^{2}$ from \cite{castro2017trisections} (see Figure \ref{arceddiagramT2D2}). The first diagram is obtained from a sutured Heegaard diagram of $M$ that induces the decomposition on the boudary of $M$: $\partial M = ( \partial M \setminus \mathring{D}) \cup D$, where $D$ is a disk in $\partial M$. An example of a sutured diagram verifying this condition is featured on Figure \ref{modelshsforMopenbook}. The condition on the sutured Heegaard splitting allows the pages of the open-book decompositions of $\big( (M \times I ) / \sim \big)$ and $T^{2} \times D^{2}$ to be diffeomorphic. Note that, because $\phi$ is the identity on a neighborhood of $\partial M$, it preserves any sutured decomposition of $\partial M$, therefore this specific sutured Heegaard splitting always exists (see Theorem \ref{th:sut}). We obtain a relative trisection diagram with page a torus with two holes and compute a system of arcs that matches the one featured on Figure \ref{arceddiagramT2D2} (an example of the resulting diagram is drawn on Figure \ref{arceddiagrambundleopen}). Then we glue the two diagrams according to the identity, thus constructing a trisection diagram for $\mathcal{OB}(M)$ (see Figure \ref{trisectiondiagramopenbook} for an example; as the genus of the surfaces involved is quite high, we used this time a more suitable planar representation of these surfaces). While constructing the system of arcs of the relative trisection diagram for the bundle, one can check that it does not depend on the monodromy, nor on $M$ (in our figures, we took the monodromy to be the identity). Therefore we have produced a $(6g+4)$--trisection diagram of $\mathcal{OB}(M)$, from a sutured Heegaard splitting of $M$ of genus $g$.    
\end{proof}

\begin{figure}[h!]
\[
\begin{tikzpicture}[scale=0.2]
\node (mypic) at (0,0) {\includegraphics[scale=0.4]{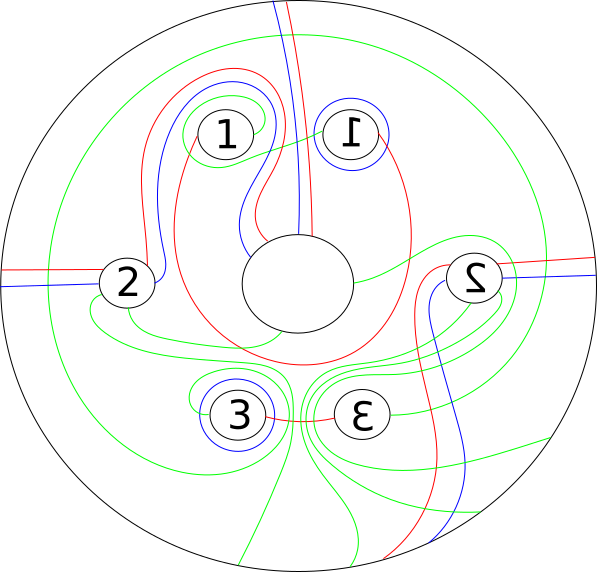}};
\end{tikzpicture}
\]
\caption{An arced diagram for $T^2 \times D^2$}
\label{arceddiagramT2D2}
\end{figure}

\begin{figure}[h!]
\[
\begin{tikzpicture}[scale=0.2]
\node (mypic) at (0,0) {\includegraphics[scale=0.4]{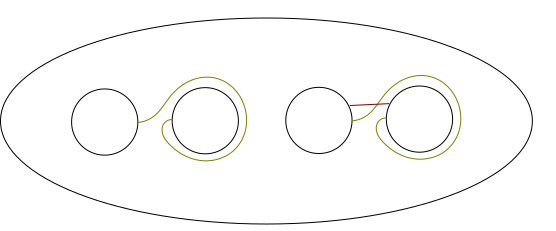}};
\node (a) at (-9,0) {A};
\node (a) at (-4,0) {\reflectbox{A}};
\node (a) at (2.5,0) {B};
\node (a) at (7.5,0) {\reflectbox{B}};
\end{tikzpicture}
\]
\caption{A sutured Heegaard splitting for some $M$ with boundary a torus\\{\footnotesize To construct such an example, consider a Heegaard diagram, with the handlebody corresponding to the $\alpha$ curves standardly embedded in $\mathbb{R}^3$,  then puncture the Heegaard surface next to an $\alpha$ curve and remove this curve.}}
\label{modelshsforMopenbook}
\end{figure}

\begin{figure}[h!]
\[
\begin{tikzpicture}[scale=0.3]
\node (mypic) at (0,0) {\includegraphics[scale=0.5]{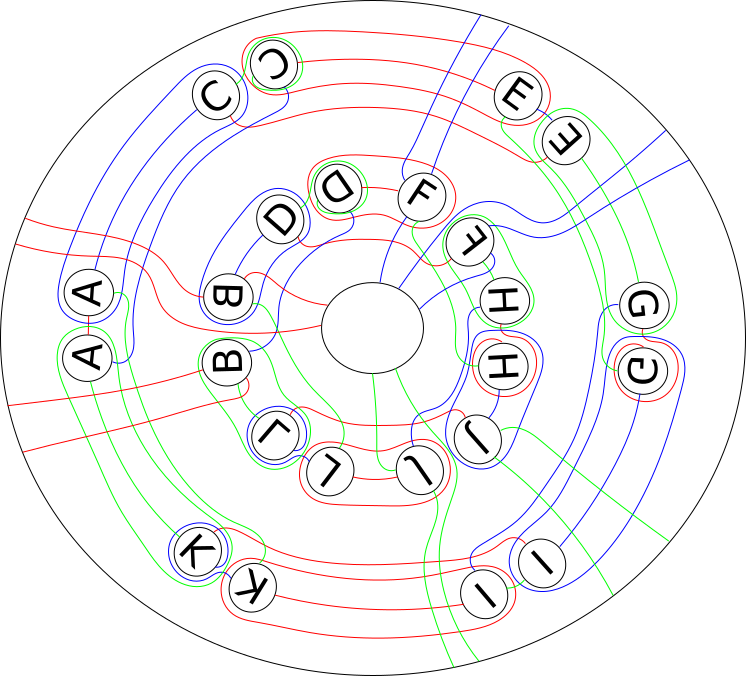}};
\end{tikzpicture}
\]
\caption{An arced diagram for $(M \times I ) / \sim$, obtained from the diagram on Figure \ref{modelshsforMopenbook}}
\label{arceddiagrambundleopen}
\end{figure}

\begin{figure}[h!]
\[
\begin{tikzpicture}[scale=0.4]
\node (mypic) at (0,0) {\includegraphics[scale=0.4]{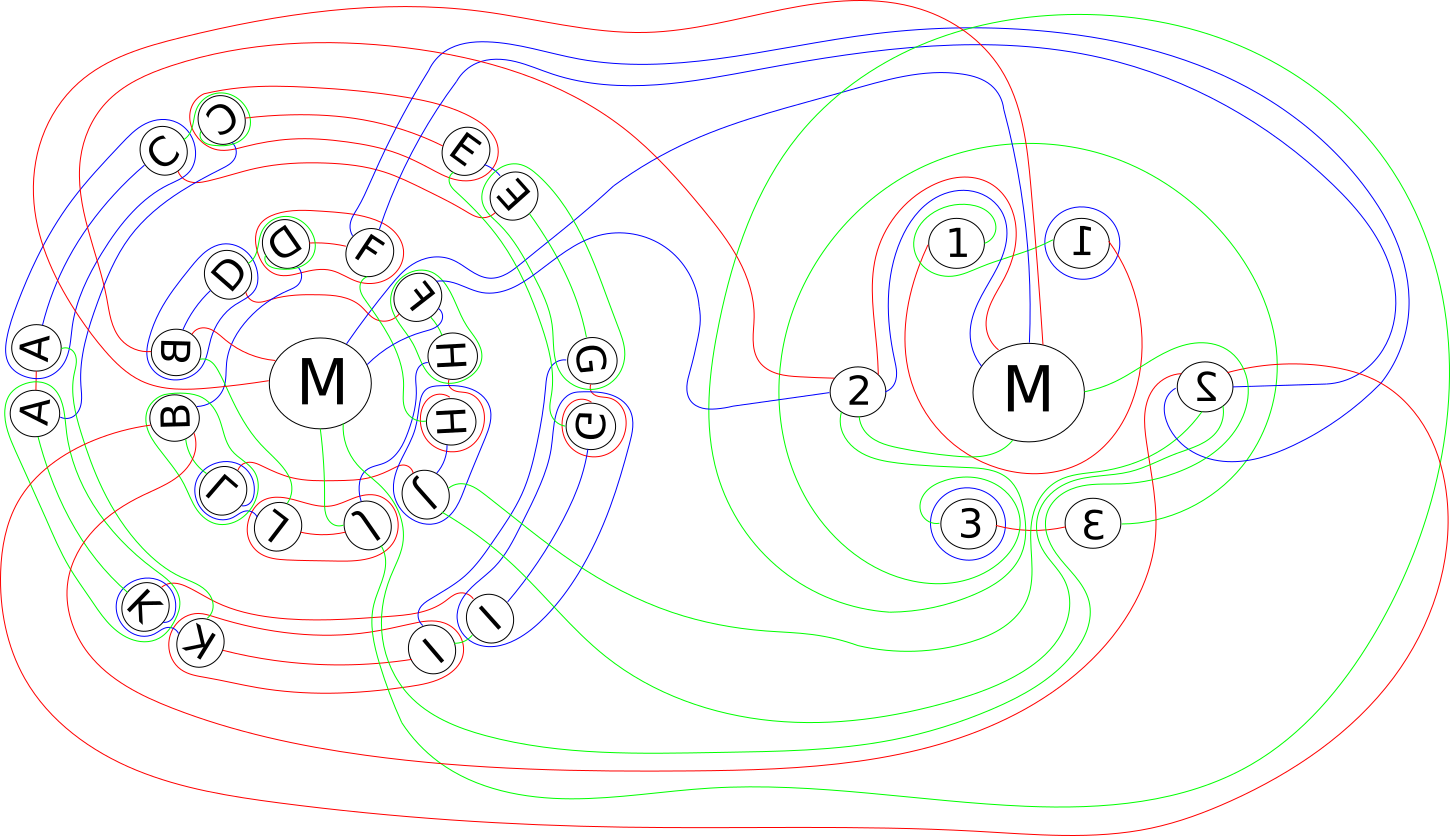}};
\end{tikzpicture}
\]
\caption{A trisection diagram for the open-book $\mathcal{OB}(M)$, obtained from the diagram on Figure \ref{modelshsforMopenbook}}
\label{trisectiondiagramopenbook}
\end{figure}

\begin{Remark}
We do not know if the genus of this trisection is minimal, nor if the diagram obtained on Figure \ref{trisectiondiagramopenbook} might be destabilised to produce a lower genus trisection.
\end{Remark}
\begin{Remark}
This method can be extended to produce a trisection diagram for \emph{any} $4$--dimensional open-book (without restricting the number of boundary components of the fiber or their genus). One has to check the compatibility of the pages induced by the relative trisections of the bundle over $S^{1}$ (of fiber $M$) and the bundle over $\partial M$ (of fiber $D^{2}$), but this can always be achieved. To see this, consider first that $\partial M$ is a connected surface of genus $h$. Then (see \cite{castro2018diagrams}), $\partial M \times D^{2}$ admits a $(h+2,2h+1;h,2)$--relative trisection. But (Theorem \ref{th1}), the bundle $(M \times I) / \sim$ admits a $(6g+5b-5, 2g +p_{1} + p_{2} +4b -3 ; p_{1} + p_{2} +b -1,2b)$--relative trisection, where $\partial M$ admits a sutured decomposition between two surfaces of genus $p_{1}$ and $p_{2}$ with $b$ boundary components. Therefore, to obtain diffeomorphic pages on the induced open-book decompositions of the boundaries of the two bundles, we need $b=1$ and $p_{1}+p_{2}=h$, which can always be achieved by starting from a sutured decomposition $\partial M = (\partial M \setminus Int(D)) \cup D$, where $D$ is a disk embedded in $\partial M$. Using Remark \ref{rmkglu}, we can compute that we have obtained a $(6g+h+3,2g+h+1)$--trisection of $\mathcal{OB}(M)$, where $h$ is the genus of $\partial M$ and $g$ the genus of the sutured Heegaard splitting of $M$. Then we could extend to the case where $\partial M$ is not connected by applying this method to every boundary component. 
\end{Remark}
\FloatBarrier
\newpage  
\bibliographystyle{/home/rudy/Documents/article/smfalpha.bst}
\bibliography{/home/rudy/Documents/article/biblioarticlefibre.bib} 
\end{document}